\documentclass[12pt]{article}
\usepackage[body={6.5in,9.0in},left=1in,top=1in,centering]{geometry}
\usepackage{amsmath,graphicx,amssymb,mathrsfs,amsthm,pstricks}
\usepackage[colorlinks=true,urlcolor=blue,linkcolor=blue,citecolor=blue,pdfstartview=FitH]{hyperref}

\usepackage{authblk}
\usepackage{graphicx}
\usepackage{amssymb}
\usepackage{epstopdf}
 \newcommand{\R}{\mathbb R}
\newcommand{\E}{\mathbb E}
\newcommand{\W}{\mathcal W}

\def\LL{\mathcal L}
\newcommand{\e}{\varepsilon}
\newcommand{\tr}{\mathrm{tr}}
\renewcommand{\sc}{\mathrm{sc}}
\newcommand{\rc}{\mathrm{rc}}
\def\d{\mathrm d}
  \usepackage[sort,comma]{natbib}
 \newcommand{\comment}[1]{}

 \allowdisplaybreaks
\numberwithin{equation}{section}
\newtheorem{theorem}{Theorem}[section]
\newtheorem{proposition}[theorem]{Proposition}
\newtheorem{lemma}[theorem]{Lemma}
\newtheorem{assumption}[theorem]{Assumption}

\theoremstyle{definition}

\newtheorem{rem}[theorem]{Remark}

\title{Quantitative Contraction Rates for McKean-Vlasov Stochastic Differential Equations with Multiplicative Noise}
\author[1]{Dan Noelck}
\affil[1]{Department of Applied Mathematics, Illinois Institute of Technology, Chicago, IL 60616. {\tt dnoelck@iit.edu},
}
\date{\today}

\begin{document}

\maketitle
\begin{abstract}
This work focuses on the quantitative contraction rates for McKean-Vlasov  stochastic differential equations (SDEs) with multiplicative noise.
 Under suitable conditions on the coefficients of the SDE, this paper derives explicit quantitative contraction rates for the convergence in  Wasserstein distances of McKean-Vlasov SDEs using the coupling method.
 The contraction results are then used to prove a  propagation of chaos uniformly in time,
  which
 provides quantitative bounds on convergence rate of interacting particle systems, and establishes exponential ergodicity for McKean-Vlasov SDEs.


\end{abstract}

\section{Introduction}

We consider the McKean-Vlasov stochastic differential equation (SDE)
\begin{equation} \label{SDE 1}
    \begin{aligned}
    \d X_t&=b(X_t,\mu_t)\d t+\sigma (X_t)\d W_t, \\
    \mu_t&=\mathcal{L}(X_t),
    \end{aligned}
\end{equation}
where $W$ is a $d$-dimensional Brownian motion,  $b: \mathbb{R}^d\times \mathcal{P}_1(\mathbb{R}^d)\to \mathbb{R}^d$, $\sigma:\mathbb{R}^d\to \mathbb{R}^{d\times d}$, and $\mu_{t} = \mathcal{L}(X_t)$ is the law of the random variable $X_t$. The precise conditions on $b$ and $\sigma$
guaranteeing
\eqref{SDE 1} has a unique strong solution will be given in what follows.
Equation \eqref{SDE 1}  is an 
SDE whose drift coefficient
depends  not only on the process itself but also on the probability distribution of the process at time $t$. Note also that the diffusion coefficient of   \eqref{SDE 1} is state-dependent. This type of equation naturally  arises as the limit of system of interacting ``particles.''  Consider, for example,  a large number of players with symmetric dynamics   whose positions depend on the positions of the mean fields of other players.
 As the number of players tends to infinity, the limit dynamic of each player does not depend on the positions of the others anymore but only on their statistical distributions thanks to the law of large numbers. The resulting limit system is a McKean-Vlasov process described by equation \eqref{SDE 1}. The study of McKean-Vlasov SDEs was
initiated in \cite{McKean-66,McKean-67} and further developed in \cite{Funaki-84,Sznitman-91}.
Since the introduction of mean-field games independently by \cite{HuangMC-06} and \cite{LasryL-07}, McKean-Vlasov SDEs have received growing attention; some of the recent developments in this direction can be found in \cite{DingQ-21,Chau-20,Wang-18} and the references therein. We also refer to the survey paper \cite{HuangRW-21} for  an extensive review of the  recent progresses in the study of McKean-Vlasov SDEs.

The investigation of convergence to equilibrium of Markov processes is of central
interest and vital importance. Crucial
questions include: When does an invariant measure exist? Is it unique? What is the convergence rate? For classical SDEs, these questions have been extensively investigated in the literature. Many different methods have been developed to approach these questions as well. In particular, the classical Harris Theorem gives geometric ergodicity under a minimization  and a Lyapunov drift conditions. We refer to \cite{Khasminskii-60,Khasminskii-12} for diffusion processes and \cite{MeynT-93II,MeynT-93III}  for general Markov processes; see also   \cite{HairerM-11,HairMS-11} for  some new perspectives of the Harris Theorem.  The  coupling method is a simple but powerful probabilistic tool for the study of convergence of Markov processes.
 The classical treatments of the coupling method for diffusion processes can be found in  \cite{LindR-86,ChenLi-89,PriolaW-06} among others and was further   developed    in  \cite{ShaoX-13,BaoYY-14,CloezH-15,NguyenY-18}   for regime-switching diffusion processes, \cite{LiangW-20,Majka-17} for stochastic differential equation driven by L\`evy processes,  and \cite{BaoYY-14,WangWYZ-22} for functional stochastic differential equations. 
 The books \cite{Chen04} and \cite{Lindvall} contain  unified and comprehensive 
treatment
of the coupling method.     

   Nevertheless, little
  is known for the convergence of McKean-Vlasov processes partly due to the delicate  and subtle distribution dependence in their dynamics.
   \cite{HammSS-21} presents a sufficient condition for the existence of an invariant measure for McKean-Vlasov SDEs,  but the issues of uniqueness and convergence rate were not addressed in the paper.
   It  should also be mentioned that  \cite{Wang-18}  derives exponential ergodicity in  Wasserstein-2 distance  for general McKean-Vlasov SDEs under certain dissipative conditions.  In contrast to \cite{Wang-18},
this work aims to {\em quantify}  the contraction rate to   the invariant measure of   \eqref{SDE 1} in   Wasserstein-1 and a multiplicative Wasserstein distances using the coupling method under
contractivity and dissipativity.
However, owing to consideration of the contraction rate we need a slightly stronger condition with the diffusion coefficient not depending on the distribution.  This allows us to construct an appropriate coupling.
While  the coupling method  has been well-developed for classical SDEs, care must be taken when dealing with McKean-Vlasov SDEs. As pointed out in \cite{EberGZ:19} and  \cite{HammSS-21}, solutions to \eqref{SDE 1} with different initial laws follow different dynamics with different drift and diffusion coefficients. Consequently, even if one can find a finite coupling time $T$, it is  not clear
how one can
construct a coupling $(X_{t}, Y_{t})$ satisfying $X_{t} = Y_{t}$ for all $t\ge T$.

We adopt the coupling method in  \cite{EberGZ:19}, which deals with classical SDEs and McKean-Vlasov SDEs with additive noise, to derive quantitative Harris-type theorems for McKean-Vlasov processes with {\em multiplicative noise} and general drift coefficients.
Using a similar
coupling method as that of \cite{EberGZ:19}
we obtain {\em explicit} quantitative contraction  rates for the  McKean-Vlasov SDE given in \eqref{SDE 1}.
The essence of our approach uses
a mixed coupling. Such a method uses
a synchronous coupling when the two marginal processes are close to each other and a reflection coupling when they    are farther apart. Such an idea was also used in \cite{Zimmer-17} for infinite-dimensional SDEs with additive noise.
Our first result, Theorem \ref{Contraction result},  shows a contraction in the Wasserstein-1 distance under
contractivity at infinity
along with certain perturbation bounds on the diffusion coefficient $\sigma$.
This result is then used to derive a  uniformly in time  propagation of chaos bound in Proposition \ref{prop-chaos} and
further obtain the existence of a unique  invariant measure and exponential ergodicity with an additional growth condition on the drift $b$ in Theorem \ref{invariant measure}.  
We next present   a local contraction result using a multiplicative Wasserstein distance that depends on the initial law  under a dissipative drift condition in  Theorem \ref{Contraction result 2}.  This contraction result   also leads to a uniform in time propagation of chaos bound along with the exponential ergodicity for the process \eqref{SDE 1}; see Proposition \ref{prop-chaos 2} and Theorem \ref{contraction to invariant measure 2}.
In   Theorems \ref{Contraction result} and \ref{Contraction result 2}, the   contraction rates can be computed from the coefficients $b$ and $\sigma$ explicitly.

 While both Theorems \ref{Contraction result} and \ref{Contraction result 2} are derived using the coupling method, it is worth pointing out   the difference between the {\em contraction at infinity} and the {\em dissipative drift conditions}.  Roughly,
  contraction at infinity pertains to the interplay among the net drift of the marginal processes, which tends to  push  them  closer to each other when they diverge far apart.
  The dissipative condition, 
  however, 
  ensures the marginal process is steered back towards the origin if it strays too far from it.

The paper is organized as follows. Section \ref{sect-prelim}    introduces the coupling for \eqref{SDE 1}, which is essential for our analyses throughout the paper. Several technical lemmas are also
presented in Section \ref{sect-prelim}.
Section \ref{sec-main} presents our main contraction results.  
These contraction results are
used to derive a propagation of chaos result as well as exponential ergodicity for \eqref{SDE 1} in Section \ref{sect-applications}.  Appendix \ref{sect-appendix} contains the proofs of several lemmas.

Throughout this paper, we denote by  $\langle \cdot, \cdot \rangle$ and $|\cdot|$ the Euclidean inner product and the corresponding norm, respectively.  Furthermore, $\|\cdot\|$ is the matrix norm induced by the Euclidean norm.
  If  $(E,\rho)$ is a Polish space, $\mathcal P(E)$ is the collection of probability measures on $E$ and $\mathcal P_{1}(E)$ is the subset of $\mathcal P(E)$ with finite first absolute moment. Furthermore, we write $\mu(|\cdot|) : = \int_{\R^{d}} |x|\mu(dx)$  if $\mu\in \mathcal P_{1}(\R^{d})$.
We   define as in \cite{Villani-09} the \textit{Wasserstein distance} of two probability measures $\mu,\nu \in \mathcal{P}(E)$  on a Polish metric space $(E,\rho)$
 by
\[\mathcal{W}_{\rho,p}(\mu,\nu)=\inf_{\pi \in \Pi(\mu,\nu)}\left[\int_{E \times E} \rho(x,y)^p\pi(\d x,\d y)\right]^{1/p},\]
where $p\in[0,\infty)$ and $\Pi(\mu,\nu)$ is the collection of coupling measures,  i.e., $\pi \in \Pi(\mu, \nu)$ for $\mu,\nu\in\mathcal{P}(E)$ if $\pi \in \mathcal P(E\times E)$, $\pi(A\times E)=\mu(A)$ and $\pi(E \times A)=\nu(A)$ for every Borel set $A\subset E$.  When $p=1$ we  simply write $\W_{\rho,1}(\mu,\nu)=\W_{\rho}(\mu,\nu)$.  In fact, the Wasserstein distance is well defined if $\rho$ is only
a semi-metric.
For the case $(E, \rho) = (\R^{d}, |\cdot|)$, we write  
\[\mathcal{W}_p(\mu,\nu)=\inf_{\pi \in \Pi(\mu,\nu)}\left[\int_{\mathbb{R}^d \times \mathbb{R}^d} |x-y|^p\pi(\d x,\d y)\right]^{1/p}, \text{ for } p\ge 1.\]

\subsection{Assumptions}

We
use the following assumptions throughout the paper.

\begin{assumption}[Generalized Lipschitz condition on $b$]\label{b lipschitz}
There exists a continuous and
  bounded function $\kappa:[0,\infty) \to \mathbb{R}$ and a constant $L_1\geq 0$ such that  
\begin{equation*}
\langle x-y,b(x,\mu)-b(y,\nu)\rangle \leq   \kappa(|x-y|)|x-y|^2+L_1\mathcal{W}_1(\mu,\nu)|x-y|, \quad \forall x,y\in \R^{d},\  \mu,\nu \in \mathcal P_{1}(\R^{d}).
\end{equation*}
\end{assumption}

\begin{assumption}[Lipschitz condition on $\sigma$]\label{sigma lipschitz} 
There exist a 
constant $L_2 \geq 0$ 
 such that
    \begin{align}
  \label{e1:sigma-lip}      \|\sigma(x)-\sigma(y)\|^2&\leq L_2|x-y|^2.
    \end{align}
 \end{assumption}

\begin{assumption}[Growth condition on $b$]\label{b growth}
There exists a constant $L_3>0$ such that
\begin{equation*}
    |b(0,\mu)|\leq L_3(1+\mu(|\cdot|)).
\end{equation*}
\end{assumption}

\begin{assumption}[Boundedness of $\sigma$  and $\sigma^{-1}$]\label{bound 1} The diffusion matrix $\sigma$ is uniformly bounded and  nondegenerate in that
\begin{equation}\label{e:sigma-nondeg}
 \|\sigma(x)-\sigma(y)\|\leq M \quad \text{ and } \quad    \|\sigma(x)^{-1}\|\leq \Lambda,\qquad  \forall x,y\in \R^{d},
\end{equation} where  $\Lambda > 0$ is a constant.
\end{assumption}

\begin{assumption}\label{bound 2*}
 The constants $M$ and $\Lambda$ in Assumption
 {\rm\ref{bound 1}} satisfy 
\begin{equation}\label{e:lam-lam-condition}
 D:= 2 \Lambda^{-1}-M>0 .
\end{equation}
\end{assumption}

\begin{rem}
  We note that for any initial condition  $ \nu_{0} \in \mathcal P_{1}(\R^{d})$, \eqref{SDE 1} has a unique strong non-explosive solution under Assumptions \ref{b lipschitz}, \ref{sigma lipschitz}, and \ref{b growth}; see, for example, \cite{Wang-18}. 
 The boundedness and nondegeneracy condition \eqref{e:sigma-nondeg} is imposed so that we can derive the exponential contractions in Section \ref{sec-main}.
\end{rem}
\begin{rem}
    The function $\kappa$ in assumption \ref{b lipschitz} may take negative values.  Our first main result will require that $\kappa$ is negative for large inputs. 
\end{rem}
Finally, since $\kappa$ is continuous, we define
\begin{equation}\label{kappa_r}
    \kappa_{r}:  = \sup\{ |\kappa(s)|: s\in [0,r]\}, \quad \forall r>0.
\end{equation}

\section{Preliminary Results}\label{sect-prelim}

We
construct a coupling similar to that in  \cite{EberGZ:19} using the reflection coupling 
 of
 \cite{LindR-86}.  We start by recalling the synchronous and reflection couplings for   the  classical SDE
\begin{equation} \label{SDE 2}
    \d X_t=b(X_t)\d t+\sigma(X_t)\d B_t.
\end{equation}

Given initial values $(x_0,y_0)\in \mathbb{R}^d\times \R^{d}$ and a $d$-dimensional Brownian motion $B_t$,
the \textit{synchronous coupling} of two  solutions to \eqref{SDE 2}
is a diffusion process $(X_t,Y_t) \in \mathbb{R}^d\times \R^{d}$ 
 solving
\begin{equation*}
    \begin{aligned}
    \d X_t&=b(X_t)\d t+\sigma(X_t)\d B_t, \ \ \ &X_0=x_0,\\
    \d Y_t&=b(Y_t)\d t+\sigma(Y_t)\d B_t, &Y_0=y_0.
    \end{aligned}
\end{equation*}

 The \textit{reflection coupling} of two  solutions to \eqref{SDE 2}
 is a diffusion process $(X_t,Y_t)$ with values in $\mathbb{R}^{2d}$ solving
\begin{equation*}
    \begin{aligned}
    \d X_t&=b(X_t)\d t+\sigma(X_t)\d B_t, \ \ \ (X_0,Y_0)=(x_0,y_0)\\
    \d Y_t&=b(Y_t)\d t+\sigma(Y_t)H_t\d B_t, \ \ \ \text{for} \ t<T, \ \quad \text{ and }\ \quad X_t =Y_t, \ \ \ \text{for} \ t\geq T,
    \end{aligned}
\end{equation*}
where $T=\inf \{t\geq0 : X_t=Y_t\}$ is the coupling time, and   the matrix $H_{t}$ is defined as $H_{t}=I-2u_{t}u_{t}^\top$ with $u_{t}=(\sigma(Y_{t}))^{-1}(X_{t}-Y_{t})/|\sigma(Y_{t})^{-1}(X_{t}-Y_{t})|$ for $t < T$.

Heuristically, the reflection coupling constructs a "mirror hyperplane" between the two processes $X_t$ and $Y_t$.  The idea is that the Brownian motion should eventually help the process hit the mirror, at which point the two processes coincide.  Once the processes coincide, we can run them together by switching to the synchronous coupling.  Notice that this idea will not work in the case where the drift $b$ has measure dependence, as we would require both $X_t=Y_t$ and $\mu_t=\nu_t$, which will not be guaranteed in the above construction.  Nevertheless, we can still use this idea to construct a combined coupling.

We now construct a coupling for \eqref{SDE 1}, which
is adapted from that 
in  \cite{EberGZ:19}.  Roughly,
the idea is to   fix some $\delta \in (0,1)$ and use the reflection coupling   when $|X_t-Y_t|\geq \delta$, and the synchronous coupling when $|X_t-Y_t|\leq \delta/2$;  in between we   use a mixture of both couplings.

To proceed,
we define Lipschitz  functions $\rc : \mathbb{R}^d \times \mathbb{R}^d \to [0,1]$ and $\sc : \mathbb{R}^d \times \mathbb{R}^d \to [0,1]$ satisfying
\begin{equation} \label{transition}
\sc^2(x,y)+\rc^2(x,y)=1.
\end{equation}
Furthermore, we fix some  $\delta \in (0,1)$ and impose $\rc(x,y)=1$ for $|x-y|\geq \delta$ and $\rc(x,y)=0$ for $|x-y|\leq \delta/2$.  Note this implies that $\sc(x,y)=0$ for $|x-y|\geq \delta$ and $\sc(x,y)=1$ for $|x-y|\leq \delta/2$.  Now, we note that under Assumptions \ref{b lipschitz}, \ref{sigma lipschitz}, and \ref{b growth},
there exists
a unique, non-explosive solution $(X_t)$ to \eqref{SDE 1} for any initial probability measure $\mu_0$ by \cite[Theorem 2.1]{Wang-18}.  Therefore, we can define for a fixed initial probability measure $\mu_0$,
\[b_{\mu_0}(t,x)=b(x,\mu_t),\] where $\mu_{t} = \LL_{X_{t}}$ is the law of $X_{t}$. Moreover, Assumption \ref{b lipschitz}  implies that $b_{\mu_{0}}$ is  Lipschitz continuous with respect to $x$. Recall that $\sigma$ is Lipschitz continuous thanks to  \eqref{e1:sigma-lip}.
Thus,  by \cite[Chapter 5, Theorem 2.5]{Karatzas-S}, for fixed probability measures $\mu_0$ and $\nu_0$, parameter $\delta>0$, and independent Brownian motions $(B_t^1)$ and $(B_t^2)$, we can construct drift coefficients $b_{\mu_0}$ and $b_{\nu_0}$ and define the coupling $(U_t)=(X_t,Y_t)$ as the unique, strong, and non-explosive solution to
\begin{equation}\label{e:coupling}
    \begin{aligned}
    \d X_t&=b_{\mu_0}(t,X_t)\d t+\sigma(X_t)[\sc(U_t)\d B_t^1+\rc(U_t)\d B_t^2],\\
    \d Y_t&=b_{\nu_0}(t,Y_t)\d t+\sigma(Y_t)[\sc(U_t)\d B_t^1+\rc(U_t)H_t\d B_t^2],
    \end{aligned}
\end{equation}
with $(X_0,Y_0)=(x_0,y_0)$ and
$H_t=I-2u_tu_t^\top,$   where  \[ \ u_{t}=\begin{cases}
   \frac{(\sigma(Y_t))^{-1}(X_t-Y_t)}{|(\sigma(Y_t))^{-1}(X_t-Y_t)|}, & \text{ if  }  X_t\neq Y_t,    \\
   \text{an arbitrary unit vector},     & \text{ if } X_t=Y_t.
\end{cases}\]
  Note that the choice of $u_{t}$ when $X_t=Y_t$ is irrelevant since $\rc(U_t)=0$ in that case.  Now note that $H_{t}H_{t}^\top=I$ and so combined with \eqref{transition}, using Levy's characterization of Brownian motion, $(X_t)$ and $(Y_t)$ solve
\begin{align}
    \d X_t&=b_{\mu_0}(t,X_t)\d t+\sigma(X_t)\d B_t, \ \ \ \ X_0=x_0, \label{XProcess}\\
    \d Y_t&=b_{\nu_0}(t,Y_t)\d t+\sigma(Y_t)\d \hat{B}_t,  \ \ \ \ \ \ Y_0=y_0, \label{YProcess}
\end{align}
with respect to the Brownian motions
\begin{align*}
    B_t&=\int_0^t \rc(U_s)\d B_s^1+\int_0^t \sc(U_s)\d B_s^2, \ \ \ \ \ \text{and}\\
    \hat{B}_t&=\int_0^t \rc(U_s)H_s\d B_s^1 + \int_0^t \sc(U_s)\d B_s^2.
\end{align*}
Assumptions \ref{b lipschitz} and \ref{sigma lipschitz} guarantee the solutions to \eqref{XProcess} and \eqref{YProcess} are pathwise unique.  Thus they coincide almost surely with the solution of \eqref{SDE 1} with respect to the Brownian motions $(B_t)$ and $(\hat{B}_t)$ above and initial values $x_0$ and $y_0$, respectively.  Hence $(X_t,Y_t)$ is a coupling for \eqref{SDE 1}.

Consider the coupling $U_t=(X_t,Y_t)$ as written in \eqref{e:coupling}. 
Now   define the process $Z_t=X_t-Y_t$ and the radial process $r_t=|Z_t|$.  Then we   have
\begin{align}\label{e:Z_t sde}
\nonumber \d Z_t&=(b_{\mu_0}(t,X_t)-b_{\nu_0}(t,Y_t))\d t+ \sc(U_t)(\sigma(X_t)-\sigma(Y_t))\d B_t^1+\rc(U_t) (\sigma(X_t)-\sigma(Y_t)H_t) \d B_t^2\\
&=\beta_t\d t+ \sc(U_t) \Delta_t\d B_t^1+ \rc(U_t) \alpha_t \d B_t^2,
\end{align} where
throughout the paper, we denote $\Delta_t=\sigma(X_t)-\sigma(Y_t)$, $\alpha_t=\sigma(X_t)-\sigma(Y_t)H_t$, and $\beta_t=b_{\mu_0}(t,X_t)-b_{\nu_0}(t,Y_t)$ for notational simplicity.
In addition,  the dynamics for the  radial  process $r_t $ is described by the following lemma whose proof is   deferred to     Appendix \ref{sect-appendix}.

\begin{lemma}\label{radial lemma}
Suppose Assumptions \ref{b lipschitz} and \ref{sigma lipschitz}, then  we have
    \begin{align}\label{radial process}
        \nonumber \d r_t=&I_{\{r_{t}  > 0 \}}\bigg[\langle e_t,\beta_t \rangle \d t + \frac{\rc^{2}(U_t)}{2r_t} (\tr(\alpha_t \alpha_t^\top)-|\alpha_t^\top e_t|^2 )\d t +\frac{\sc^{2}(U_t)}{2r_t} (\tr(\Delta_t \Delta_t^\top)-|\Delta_t^\top e_t|^2 )\d t\\ &+\sc(U_t)\langle e_t,\Delta_t \d B_t^1 \rangle + \rc(U_t) \langle e_t, \alpha_t\d B_t^2 \rangle\bigg]
    \end{align}
almost surely, where \begin{displaymath}
e_{t} = \begin{cases}
 \frac{ Z_t}{r_t},    & \text{ if  } r_{t} > 0, \\[1ex]
  \frac{b_{\mu_0}(t,X_t)-b_{\nu_0}(t,Y_t)}{|b_{\mu_0}(t,X_t)-b_{\nu_0}(t,Y_t)| },    & \text{ if }r_{t } =0 \text{ and }b_{\mu_0}(t,X_t)-b_{\nu_0}(t,Y_t) \neq 0,\\[1ex]
  e, & \text{ otherwise},
\end{cases}
\end{displaymath} in which  $e$ 
is an arbitrary unit vector.
\end{lemma}

The following lemma will be used to bound the drift
when $r_t$ is small. The proof of the lemma,  which consists of elementary calculations,    is arranged in   Appendix \ref{sect-appendix}.

\begin{lemma}\label{A bound}
For $x\neq y\in \R^d$, denoting   $z=x-y$, $e=z/|z|$,  
 $u(x,y)=(\sigma(y)^{-1}z)/|\sigma(y)^{-1}z|$,  $H(x,y)=I-2u(x,y)u(x,y)^\top$,   $\Delta(x,y)=\sigma(x)-\sigma(y)$, and $\alpha(x,y)=\sigma(x)-\sigma(y)H(x,y)$. Then  
\begin{equation}\label{e1-lem3.1}
    \tr(  \alpha (x,y)\alpha(x,y)^\top)-| \alpha(x,y)^\top e|^2=\tr(\Delta(x,y) \Delta(x,y)^\top)-|\Delta(x,y)^\top e|^2.
\end{equation}
In addition, under  Assumption \ref{sigma lipschitz}, we have 
\begin{equation}\label{e2-lem3.1}
   \big| \tr(\Delta(x,y) \Delta(x,y)^\top)-|\Delta(x,y)^\top e|^2\big| \le 2d L_{2} |x-y|^2.
\end{equation} \end{lemma}


Using Lemma \ref{A bound} with Assumptions \ref{sigma lipschitz} and \ref{bound 1}, we can define a bound $A$ by
\begin{equation}\label{A definition}
A: =\sup_{x\neq y}\bigg\{\frac{d\|\sigma(x) - \sigma(y)\|^{2} |x-y|^{2}- |(\sigma(x) - \sigma(y))^{\top}(x-y)|^{2}}{|x-y|^{3}} \bigg\}
\end{equation}
where $d$ is the dimension.
Apparently, $A \ge 0$.   We also  claim that $A<\infty$. Indeed,  
since $\|\sigma(x)-\sigma(y)\| \leq M$ thanks to Assumption \ref{bound 1},  
 the expression $\frac{d|\sigma(x) - \sigma(y)|^{2} |x-y|^{2}- |(\sigma(x) - \sigma(y))^{\top}(x-y)|^{2}}{|x-y|^{3}}$ in \eqref{A definition} is bounded when $|x-y|$ is large. On the other hand,   Lemma \ref{A bound} ensures  that it is also  bounded when $|x-y|$ approaches $0$. Therefore the claim that  $A<\infty$ follows. 


\section{Main Results}\label{sec-main}

This section  gives two different contraction rates for two solutions starting with different initial values.  The first result assumes that as the two processes move far away from each other, the net drift between them
is strong enough to push them back together.  In this sense we ensure that the function $\kappa$ is negative enough as $r$ gets large.  The second result
assumes that the drift pushes radially inward when the process is far
away from the origin.
Here we ensure that the coefficient $b$ is negative enough when $X_t$ is large.

\subsection{Contraction Result Under  Contractivity at Infinity}\label{result 1}

Our first result needs
the following  assumption.


\begin{assumption}\label{limsup kappa}
The function $\kappa(r)$ in Assumption {\rm \ref{b lipschitz}} satisfies
\begin{equation} \label{limsup condition}
    \limsup_{r\to \infty}\kappa(r)< 0. 
\end{equation} 
\end{assumption}

  In order to construct the  metric $\rho$  to be used 
   in Theorem \ref{Contraction result},  
   we   define the constants $R_{1}$ and   $R_2$ to be
\begin{align*}
   R_{1} & : =    1 \vee \inf\{R\ge 0: \kappa(r) r < -A , \ \forall r\ge R\}, \\
    R_2&: =\inf\left\{R>R_1 : \kappa(r) \leq -\left(\frac{D^2}{R(R-R_1)}+\frac{A}{R}\right),\  \ \forall r\geq R\right\},
\end{align*} where $A$ is the nonnegative constant defined in \eqref{A definition} and $D$ is defined in \eqref{e:lam-lam-condition}. Note that
$R_{1}, R_2<\infty$ by \eqref{limsup condition}.
  Next we define \begin{displaymath}
\varphi(r)=\exp\left(-\frac{2}{D^2}\int_0^r(u\kappa^*(u)+A)\d u\right), \ \text{ and } \ \Phi(r)=\int_0^r\varphi(s)\d s, \quad \forall r\ge 0,
\end{displaymath} where $\kappa^{*}$ is defined by $\kappa^*(r)=\kappa(r)$ when $r\kappa(r)\geq -A$ and $\kappa^*(r)=-A/r$ when $r\kappa(r)<-A$.  Clearly we have $\kappa^*(r)\geq \kappa(r)$ and $r \kappa^*(r) +A \geq 0$ for all $r\ge 0$.  Hence $\varphi(r)\leq 1 $ for $r\geq 0$.  Moreover,  $\varphi(r)  = \varphi(R_{1})$  and hence $\Phi(r) = \Phi(R_{1}) + (r-R_{1}) \varphi(R_{1})$   for $r \ge R_{1}$. Now, we define the constant $c > 0$ to be 
\begin{equation}\label{e:c-def}
    c:=\dfrac{D^{2}}  {4 \int_0^{R_2}\Phi(s)\varphi(s)^{-1}\d s}
\end{equation}
Next, we consider the function \begin{displaymath}
g(r)=1-\frac{2c}{D^2}\int_0^r\Phi(s)\varphi(s)^{-1}\d s, \ \ \forall r\ge 0.
\end{displaymath} Note that $\frac{1}{2}\leq g(r\wedge R_2)\leq 1$. Finally, the function $f:[0,\infty) \to [0,\infty)$ is given by\begin{equation}
\label{e:f-defn}
 f(r)=\int_0^r\varphi(s)g(s\wedge R_2)\d s, \ \ \forall r\ge 0.
\end{equation}
Note that $f(0)=0$, $f(r)>0$ for $r>0$,  and $f \in C^{1}(0,\infty) \cap C^{2}((0,\infty) \setminus\{R_{2}\})$ with \begin{align}
  \label{f 1st derivative}  & f'(r)=\varphi(r)g(r\wedge R_{2}) \in (0, 1],\\
 \label{f second derivative}& \begin{aligned}  f''(r)& =  -\frac{2f'(r)}{D^2}(r\kappa^*(r)+A)-\frac{2c}{D^2}\Phi(r)I_{\{r < R_{2} \}} \\ & \leq -\frac{2f'(r)}{D^2}(r\kappa^*(r)+A)-\frac{2c}{D^2}f(r) I_{\{r < R_{2} \}} \le 0.\end{aligned}
\end{align}
Consequently, $f$ is a strictly increasing and concave function.
Note also that  $f$ is linear on $[R_2,\infty)$.

Finally, it is clear
that $f(|x-y|) = f(|y-x|) $ and $f(|x-y|)  \le f(|x-z|) + f(|z-y|)  $ for all $x,y,z\in \R^{d}$.   Thus,  \begin{equation}
\label{e:rho-defn}
\rho(x,y): =f(|x-y|)
\end{equation} is a well defined metric on $\R^{d}$.

The following lemma gives linear bounds on the function $f$.

\begin{lemma}\label{linear bounds}
For the functions $f$, $\varphi$, and $\Phi$ defined above, we have
\[r\varphi(R_1)\leq \Phi(r) \leq \frac{D^2f(r)}{2} \leq \frac{D^2\Phi(r)}{2}\leq \frac{D^2r}{2}.\]
\end{lemma}

\begin{proof}
The first inequality follows since $\varphi(r)$ is decreasing for $r<R_1$ and $\varphi(r)=\varphi(R_1)$ for $r\geq R_1$, so
\[r\varphi(R_1)-\Phi(r)=\int_0^r(\varphi(R_1)-\phi(s))\d s\leq 0.\]
Next, the observation  $2/D^2\leq g(r\wedge R_2)\leq 1$  gives the second and the third inequalities.  The last inequality follows since $\varphi(r)\leq1$.
\end{proof}

\begin{theorem}\label{Contraction result}
Suppose Assumptions {\rm\ref{b lipschitz}}, {\rm\ref{sigma lipschitz}}, {\rm\ref{b growth}}, {\rm\ref{bound 1}},  {\rm\ref{bound 2*}},  and {\rm\ref{limsup kappa}} hold.   
Let $\mu_t$ and $\nu_t$ denote the marginal laws of a strong solution $(X_t)$ of equation \eqref{SDE 1} with initial conditions $X_0 \sim \mu_0\in \mathcal P_{1}(\R^{d})$ and $X_0 \sim \nu_0\in \mathcal P_{1}(\R^{d}) $, respectively. Then there exist 
 constants $C$ and $\gamma$ such that
\begin{align}
    \mathcal{W}_{\rho}(\mu_t,\nu_t) &\leq e^{-\gamma t}\mathcal{W}_{\rho}(\mu_0,\nu_0) \label{Wrho}\\
    \mathcal{W}_1(\mu_t,\nu_t) &\leq Ce^{-\gamma t}\mathcal{W}_1(\mu_0,\nu_0) , \label{W1}
\end{align}
where the metric $\rho$ is given in \eqref{e:rho-defn} and 
the constants are given by
\begin{align}\label{e:gamma-defn}
    \gamma = c-\frac{L_1D^2}{2\varphi(R_1)}\ \ \ \text{ and }  \ \ \
    C =\frac{D^2}{2\varphi(R_1)},\end{align}
in which  $c$, $R_1$, $D$ and $\varphi$ are defined as above.

\end{theorem}

\begin{rem}
    In application, one would hope for $\gamma>0$.  This requires $L_1,$ to be small when $R_1$ and $R_2$ are large. Such requirements are, in fact, natural. We notice that $\kappa(r)<0$ is needed to ensure recurrence of the radial process inside some ball of radius $R_2$.  However, if $L_1$ is large, one could have multiple distinct stationary solutions to \eqref{SDE 1} inside that ball.
\end{rem}



\begin{proof}
We consider the function $f$ defined in \eqref{e:f-defn} and note that using the coupling $U_t=(X_t,Y_t)$ from Section \ref{sect-prelim}, in which we assume without loss of generality that $\mathcal{W}_{\rho}(\mu_0,\nu_0) = \E[\rho(X_{0}, Y_{0})]= \E[f(r_{0})]$, we have
\[\mathcal{W}_{\rho}(\mu_t,\nu_t)\leq \E[\rho(X_t,Y_t)]=\E[f(|X_t-Y_t|)] = \E[f(r_{t})].\]
Therefore, we need to estimate $\E[f(r_{t})]$, in which $r_{t}$ satisfies \eqref{radial process}.
Applying the It\^o-Tanaka  formula (see, for example, \cite[Theorem 22.5]{Revuz-Yor-99}) to $f(r_{t})$, 
we get almost surely
\begin{align}\label{Ito-Tanaka}
 \nonumber       f(r_t)=&\ f(r_0)+\int_0^t f'(r_s)\d r_s+\frac{1}{2}\int_{-\infty}^\infty L_s^x \mu_f(\d x) \\
   \nonumber     =&\ f(r_0)+\int_0^t   I_{\{r_{s} > 0 \}} f'(r_s)\langle e_s, \beta_s \rangle \d s +\int_0^t  I_{\{r_{s} > 0 \}} f'(r_s)\frac{\rc^2(U_s)}{2r_s}\left(\tr(\alpha_s\alpha_s^\top)-|\alpha_s^\top e_s|^2\right)\d s\\
        &+\int_0^t  I_{\{r_{s} > 0 \}} f'(r_s)\frac{\sc^2(U_s)}{2r_s}\left(\tr(\Delta_s \Delta_s^\top )-|\Delta_s^\top e_s|^2\right)\d s \\
 \nonumber       &+\int_0^t  I_{\{r_{s} > 0 \}} f'(r_s) [\rc(U_s)\langle e_s,\alpha_s \d B_s^2 \rangle +\sc(U_s)\langle e_s,\Delta_s\d B_s^1 \rangle]+\frac12 \int_{-\infty}^\infty L_t^x\mu_f(\d x), 
    \end{align}
where $L_t^x$ is the right-continuous local time of $(r_t)$, 
and $\mu_f$ is the nonpositive measure representing the second derivative of $f$.  Furthermore, from the It\^o-Tanaka formula we have for every measurable function $g:\mathbb{R} \to \mathbb{R}_+$
\[\int_0^tg(r_s)\d[r]_s=\int_{-\infty}^\infty g(x)L_t^x\d x.\]
Thanks to  \eqref{e:sigma-nondeg} and \eqref{e:lam-lam-condition}, 
we can use the same argument as that in \cite[section 3]{LindR-86} to show that \begin{equation}
\label{e:alpha e>D>0}
|\alpha_{s}^\top e_{s}|^2\geq D^2>0.
\end{equation}
  Since $\delta< 1 \le  R_1$, it follows that  $r_{s} \ge \delta  $ if $ I_{\{r_s\in \{R_1,R_2\}\}} =1$. Moreover, on the set $\{r_{s} \ge \delta \}$, $Z_s\neq0$, and we have $\sc^2(U_s)=0$ and $\rc^2(U_s)=1$ therefore, 
  \[\d [r]_{s} =I_{\{Z_s \neq 0 \}}[ \rc^2(U_s)|\alpha_s^\top e_s|^2+\sc^2(U_s)|\Delta_s^\top e_s|^2]\d s  = |\alpha_s^\top e_s|^2 \d s  \ge D^{2} \d s .\]
Consequently, we have
\begin{align*}
\text{Leb}\{0\leq s \leq t : r_s\in \{R_1,R_2\}\} &= \int_0^t I_{\{r_s\in \{R_1,R_2\}\}}\d s
 \le \frac{1}{D^{2}}  \int_0^t   I_{\{r_s\in \{R_1,R_2\}\}} \d[r]_{s}\\
&=\frac{1}{D^{2}} \int_{\{R_1,R_2\}}L_t^x\d x=0 \ \ \ \text{a.s.}\end{align*}
where $ \text{Leb}\{0\leq s \leq t : r_s\in \{R_1,R_2\}\}$ is the Lebesgue measure.  Furthermore, recall that  $f$ is concave  and that  $f\in C^{1}(0, \infty) \cap C^{2}((0,\infty) \setminus \{R_1,R_2\})$. Hence,
\begin{equation}\label{Local Time}\begin{aligned}
\int_{-\infty}^\infty L_t^x\mu_f(\d x) &\leq \int_0^t f''(r_s)\d[r]_s\\&=\int_0^t f''(r_s)\left(\rc^2(U_s)|\alpha_s^\top e_s|^2+\sc^2(U_s)|\Delta_s^\top e_s|^2\right)\d s \ \text{ a.s.}\end{aligned}\end{equation}
We now plug \eqref{Local Time} into   \eqref{Ito-Tanaka} to obtain
\begin{align} \label{df(r)}
 \nonumber  &  \d f(r_t) \\\nonumber  &\leq   I_{\{r_{t} > 0 \}}\bigg[f'(r_t)\bigg(\!\langle e_t,\beta_t \rangle + \frac{\rc^2(U_t)}{2r_t}(\tr(\alpha_t \alpha_t^\top )-|\alpha_t^\top e_t|^2)+\frac{\sc^2(U_t)}{2r_t}(\tr(\Delta_t \Delta_t^\top )-|\Delta_t^\top  e_t|^2)\bigg)\bigg]\d t\\
    &\quad +I_{\{r_{t} > 0 \}} \frac{f''(r_t)}{2}\left(\rc^{2}(U_t)|\alpha_t^\top  e_t|^2+ \sc^{2}(U_t)|\Delta_t^\top  e_t|^2\right)\d t \\
 \nonumber   &\quad + I_{\{r_{t} > 0 \}} f'(r_t)\left(\rc^{2}(U_t)\langle e_t, \alpha_t \d B^2 _t \rangle + \sc^{2}(U_t) \langle e_t, \Delta_t \d B^1_t\rangle\right).
    \end{align}

Using integration by parts and \eqref{df(r)}, we have   for the  constant $\gamma= c-\frac{L_1D^2}{2\varphi(R_1)}$,
\begin{equation}\label{expectation of exponential}
    \E[e^{\gamma t}f(r_t)-f(r_0)]\leq \int_0^t e^{\gamma s}\E[\mathcal{E}_s]\d s,
\end{equation}
where
\begin{align} \label{E}
 \nonumber
    \mathcal{E}_s:=&\, \gamma f(r_s)+f'(r_s)I_{\{r_{s} > 0 \}} \langle e_s, \beta_s \rangle+ \frac{f''(r_s)}{2}I_{\{r_{s} > 0 \}} (\rc^{2}(U_s)|\alpha_s^\top  e_s|^2+\sc^{2}(U_s)|\Delta_s^\top  e_s|^2) \\
    &+I_{\{r_{s} > 0 \}} \frac{f'(r_{s})}{2r_s}\big[ \rc^2(U_s) \left(\tr(\alpha_s \alpha_s^\top )-|\alpha_s^\top  e_s|^2\right) +\sc^{2}(U_{s}) \left(\tr(\Delta_s \Delta_s^\top )-|\Delta_s^\top  e_s|^2\right)\big].
    \end{align}

We now attempt to bound $\mathcal{E}_s$ so that $\E[\mathcal{E}_s]\leq (\delta)$.  To that end we will first find bounds for each term and substitute in the choice of $\gamma$.
 First, assumption \ref{b lipschitz} along with Lemma \ref{linear bounds} and the fact that $\mathcal{W}_1(\mu_{s},\nu_{s})\leq \E[r_{s}]$ yields the bound
\begin{equation}
\label{e1:e-betas-estimate}
\begin{aligned}   I_{\{r_{s} > 0 \}} \langle e_s, \beta_s \rangle &  \leq I_{\{r_{s} > 0 \}}( r_s\kappa(r_s)+L_1\mathcal{W}_1(\mu_s,\nu_s))\\ & \leq I_{\{r_s\geq \delta\}}r_s\kappa(r_s)+I_{\{r_s<\delta\}}\kappa_{\delta}\delta
+\frac{L_1D^2}{2\varphi(R_1)}\E[f(r_s)],\end{aligned}
\end{equation}
where $\kappa_r$ is define in \eqref{kappa_r}.
 Since $\sc^2(U_s)+\rc^2(U_s)= 1$ and $f'(r_s)\leq 1$, using Lemma \ref{A bound}, we get on the set $\{0 < r_s<\delta\}$
\begin{equation}\label{K_0 ref}
\begin{aligned}
&\frac{f'(r_{s})}{2r_s}\big[ \rc^2(U_s) \left(\tr(\alpha_s \alpha_s^\top )-|\alpha_s^\top  e_s|^2\right) +\sc^{2}(U_{s}) \left(\tr(\Delta_s \Delta_s^\top )-|\Delta_s^\top  e_s|^2\right)\big]  \\
 &\ \ \le  \frac{f'(r_{s})}{2r_s} \big[ \rc^2(U_s) \cdot 2 d L_{2} r_{s}^2 +   \sc^2(U_s)\cdot  2 d L_{2} r_{s}^{2}\big]  = dL_{2} f'(r_{s}) r_{s} \le dL_{2} \delta.
\end{aligned}
\end{equation}
Also recall  that  $\sc(U_s)=0$ and $\rc(U_s)=1$ for $r_s\geq \delta$. Therefore,  using \eqref{e1-lem3.1} in Lemma \ref{A bound}   and the definition of $A$ in \eqref{A definition}, we get on the set $\{r_s \geq \delta\}$,
\begin{equation}\label{A bound2}
\frac{\tr(\alpha_s \alpha_s^\top )-|\alpha_s^{\top } e_s|^2}{2r_s} = \frac{\tr(\Delta_s \Delta_s^\top )-|\Delta_s^{\top } e_s|^2}{2r_s} \le \frac{d\|\Delta_s\||r_s|^2-|\Delta^\top r_s|^2}{2r_s^3}\leq\frac{ A}{2}.
\end{equation}
In the above we used the bound $\tr (\Delta \Delta)^\top\leq rank(\Delta)|\Delta|^2\leq d|\Delta|^2$.
Now plugging  \eqref{e1:e-betas-estimate}, \eqref{K_0 ref}, and \eqref{A bound2} into   \eqref{E}, we obtain 
\begin{align} \label{E bound}
  \nonumber  \mathcal{E}_s\leq&\, \frac{L_1D^2}{2\varphi(R_1)}(\E[f(r_s)]-f(r_s)) 
   +I_{\{r_s<\delta\}}\left(cf(r_s) + \kappa_{\delta}\delta + dL_{2} \delta \right)   \\
  \nonumber   &+I_{\{r_s\geq \delta\}}\left(cf(r_s) + f'(r_s)\left(r_s\kappa(r_s)+\frac{\tr(\Delta_s \Delta_s^\top )-|\Delta_s^{\top } e_s|^2}{2r_s}\right)+\frac{f''(r_s)}{2}|\alpha_s^\top  e_s|^2\right)\\
    \leq&\,  \frac{L_1D^2}{2\varphi(R_1)}(\E[f(r_s)]-f(r_s))
     + I_{\{r_s <\delta\}}(c +\kappa_{\delta} +dL_{2} ) \delta\\
 \nonumber   &+ I_{\{r_s\geq \delta\}}\bigg(cf(r_s)+f'(r_s)\bigg(r_s \kappa(r_s)+\frac{A}{2}\bigg)+\frac{D^2}{2}f''(r_s)\bigg),
    \end{align} where we used \eqref{e:alpha e>D>0}, \eqref{A bound2}, to derive the last inequality.

Note that the first term on the right-hand side of \eqref{E bound} is zero in expectation.  The second term goes to zero as $\delta \to 0$.  Therefore if we can show the last term is  nonpositive,  we will have $\E[\mathcal{E}_{s}] \leq O(\delta)  $. 
To this end, we  have from \eqref{f second derivative} that
\begin{equation}
\label{e1:3.30}
\begin{aligned}& I_{\{r_{s} < R_{2} \}} \bigg[cf(r_s)+f'(r_s)\bigg(r_s\kappa(r_s)+\frac{A}{2}\bigg)+\frac{D^2}{2}f''(r_s) \bigg] \\ & \ \ \le  I_{\{r_{s} < R_{2} \}} \bigg[cf(r_s)+f'(r_s)(r_s\kappa^{*}(r_s)+A)+\frac{D^2}{2}f''(r_s) \bigg]\leq 0.\end{aligned}
\end{equation}
Next, using the facts that $\varphi(s) = \varphi(R_{1})$ and $\Phi(s) = \Phi(R_{1}) + \varphi(R_{1})(s-R_{1})$ for $s\ge R_{1}$, we can show 
\begin{equation}\label{e-PhiR2} 
 \int_{0}^{R_2}\Phi(s)\varphi(s)^{-1}\d s\geq \int_{R_1}^{R_2}\Phi(s)\varphi(s)^{-1}\d s \geq \frac{\Phi(R_2)(R_2-R_1)}{2\varphi(R_1)}.\end{equation} On the other hand, since $\Phi$ is increasing and concave with $\Phi(0) =0$, on the set $\{r_{s} \ge R_{2} \}$, we have $\frac{\Phi(r_{s})}{r_{s}} \le \frac{\Phi(R_{2})}{R_{2}} $ and hence \begin{equation}\label{e:Phi-concave}
\frac{r_{s}}{R_{2}} \ge \frac{\Phi(r_{s})}{\Phi(R_{2})}.
\end{equation}
Finally, we  note that for $r\geq R_2$, $f''(r)=0$  
and $f'(r)=\varphi(R_1)g(R_2)=\frac{2\varphi(R_1)}{D^2}$.  Now using the definition of $R_2$, we get on the set $\{r_{s} \ge R_{2} \}$
\begin{align*}
    f'(r_s)( r_s \kappa(r_s)+A) &=\frac{2\varphi(R_1)}{D^2}(r_s\kappa(r_s)+A)
    \leq \frac{2\varphi(R_1)}{D^2}\left(-r_s\left(\frac{D^2}{R_2(R_2-R_1)}+\frac{A}{R_2}\right)+A\right)\\
    &\leq -\frac{2\varphi(R_1)}{D^2}\frac{D^2r_s}{R_2(R_2-R_1)}
    = \frac{-2r_s\varphi(R_1)}{R_2(R_2-R_1)}
    \leq \frac{-2\varphi(R_1)\Phi(r_s)}{\Phi(R_2)(R_2-R_1)}\\
    &  \le  - \frac{\Phi(r_s)}{ \int_{0}^{R_2}\Phi(s)\varphi(s)^{-1}ds}  \le -\frac{4c}{D^2} \Phi(r_s) \le - \frac{4c}{D^2} f(r_{s}),
\end{align*}
where we used \eqref{e:Phi-concave}, followed by \eqref{e-PhiR2}, followed by the definition of $c$ \eqref{e:c-def}, and finally lemma \ref{linear bounds} to derive the last four  inequalities. 
This together with \eqref{f second derivative} imply that \begin{equation}
\label{e2:3.30}\begin{aligned}
& I_{\{r_{s} \ge R_{2} \}} \bigg[cf(r_s)+f'(r_s)\bigg(r_s\kappa(r_s)+\frac{A}{2}\bigg)+\frac{D^2}{2}f''(r_s) \bigg] \\ &\ \  \le I_{\{r_{s} \ge R_{2} \}}[cf(r_s)+f'(r_s)(r_s\kappa(r_s)+A)] \leq 0.\end{aligned}
\end{equation}

Using \eqref{e1:3.30} and \eqref{e2:3.30} in \eqref{E bound} and then taking expectations, we obtain $\E[\mathcal E_{s}] \le (c + \kappa_{\delta} + d L_{2}) \delta $; which, together with
  \eqref{expectation of exponential},  gives us
\begin{equation*}
    \E[e^{\gamma t}f(r_t)-f(r_0)]\leq (c+\kappa_{\delta}+dL_{2})\delta \frac{e^{\gamma t}}{\gamma}.
\end{equation*}
Then it follows that
\begin{align*}
    \mathcal{W}_\rho(\mu_t,\nu_t) &        \leq \E[f(r_t)]\leq e^{-\gamma t} \E[f(r_{0})] + \frac{  (c+\kappa_{\delta}+dL_{2})\delta}{\gamma}\\
     &    = e^{-\gamma t}\mathcal{W}_\rho(\mu_0,\nu_0)  + \frac{  (c+\kappa_{\delta}+dL_{2})\delta}{\gamma}.
\end{align*}  Recall that $\gamma$ is independent of the coupling, so it does not depend on $\delta$.  Therefore, since $\delta > 0$ can be chosen arbitrarily small,   \eqref{Wrho} follows.

Finally,   Lemma \ref{linear bounds} leads to
\begin{equation}\label{e:f(r_t)estimate}
    \mathcal{W}_1(\mu_t,\nu_t)\leq \frac{D^2}{2\varphi(R_1)}\E[f(r_{t})]\leq \frac{D^2}{2\varphi(R_1)}e^{-\gamma t}\mathcal{W}_\rho (\mu_0,\nu_0) \leq \frac{D^2}{2\varphi(R_1)}e^{-\gamma t}\mathcal{W}_1(\mu_0,\nu_0),
\end{equation}
which shows \eqref{W1}, completing the proof.
\end{proof}

\subsection{Contraction Result Under Dissipative Condition}

We now replace Assumption \ref{limsup kappa} with the following assumption.

\begin{assumption}[Drift]\label{drift}
There exist positive constants $R$ and $  \lambda$ such that 
\begin{equation*}
    \langle x, b(x,\mu) \rangle \leq -\lambda |x|^2 + L_4|x|\mu(|\cdot |), \quad \forall |x| \ge R.
\end{equation*}

\end{assumption}

\begin{theorem}\label{Contraction result 2}
Assume assumptions {\rm\ref{b lipschitz}, \ref{sigma lipschitz}, \ref{b growth}, \ref{bound 1}, \ref{bound 2*}}, and {\rm\ref{drift}} with $L_4\leq L_1 \leq \lambda/2$.  Let $\mu_t$ and $\nu_t$ denote the marginal laws of a strong solution $(X_t)$ to \eqref{SDE 1} with initial laws $\mu_{0}, \nu_{0}$, respectively.  Then there exist a concave, bounded, nondecreasing continuous function $f:\R_+ \to \R_+$ with $f(0)=0$, a Lyapunov function $V(x)$ and constants $c,\e,K_1,K_6 \in (0,\infty)$ such that
\begin{enumerate}
    \item[{\rm(i)}] For any $K\in (0,\infty)$ there is an $L_1^*>0$ such that for any $L_1<L_1^{*}$ if the initial laws have $\mu_0(V),\nu_0(V)\leq K$, then
    \begin{align}
        \W_{\rho_1}(\mu_t,\nu_t)&\leq e^{-ct}\W_{\rho_1}(\mu_0,\nu_0), \ \ \ \text{and} \label{contraction 1}\\
        \W_{1}(\mu_t,\nu_t) &\leq K_{1}e^{-ct}\W_{\rho_1}(\mu_0,\nu_0). \label{contraction 2}
    \end{align}
    \item[{\rm(ii)}] There is an $L_1^{**}>0$ such that for any $L_1<L_1^{**}$ and initial laws $\mu_0,\nu_0.$
    \begin{equation}
        \W_{\rho_1}(\mu_t,\nu_t)\leq e^{-ct}(\W_{\rho_1}(\mu_0,\nu_0)+K_6[\e \mu_{0}(V)+\e \nu_{0}(V)]^2), \label{contraction 3}
    \end{equation}
\end{enumerate}
where $ c=\frac12 \min\{\lambda/2,\eta D^2/8\}$ with $\eta$ given  below in \eqref{e:eta-defn-result2}, and the semi-metric $\rho_{1}$ is given by $$\rho_1(x,y)=f(|x-y|)(1+\e V(x)+\e V(y)).$$
\end{theorem}

\begin{rem}
Note again the Lipschitz coefficient for the law dependence, $L_1$, needs to be small enough.  The choice of $c$ utilizes the rate, $\lambda$, at which the process returns to the ball, as well as the rate of convergence inside the ball due to the coupling.
\end{rem}
We will construct a function $f$ similar to that from Section \ref{result 1}, but we will first need the following lemma.

\begin{lemma}\label{dV lemma}
Let $V(x)=1+|x|^2$.  Suppose  Assumptions
{\rm\ref{b lipschitz}}, {\rm\ref{sigma lipschitz}}, {\rm\ref{b growth}},
and {\rm\ref{drift}} hold, and  $L_4 \leq L_1 \leq \frac{\lambda}{2}$.  Then there exists a constant $L \in (0,\infty)$ such that
\begin{equation}\label{dV(X_t) bound}
   \d V(X_t)\leq (L-\lambda V(X_t)-\lambda |X_t|^2+2L_1|X_t|\E[|X_t|])\d t + 2 \langle X_t, \sigma(X_t)\d B_t \rangle,
\end{equation}
and
\begin{equation}\label{E[V(X_T) bound}
    \E[V(X_t)]\leq \frac{L}{\lambda}+e^{-\lambda t}\E[V(X_0)].
\end{equation}
\end{lemma}

\begin{proof}
By It\^o's formula with Assumptions \ref{bound 1} and \ref{drift} and the fact that $L_4 \leq L_1$, we have
\begin{align*}
    \d V(X_t) & = \left(2\langle X_t,b(X_t,\mu_t) \rangle + \tr(\sigma^\top (X_t)\sigma(X_t))\right)\d t +2 \langle X_t,\sigma(X_t)\d B_t \rangle \\
  &   \leq   I_{\{|X_t|\leq R\}} 2 \langle X_t, b(X_t,\mu_t) \rangle \d t + I_{\{|X_t|\geq R\}}2   (-\lambda |X_t|^2+L_1|X_t|\E[|X_t|]  )\d t \\
    & \quad + dM \d t+ 2\langle X_t,\sigma(X_t)\d B_t\rangle.
\end{align*}
We next recall the definition of $\kappa_R$ from \eqref{kappa_r} and use Assumption \ref{b lipschitz} to compute \begin{align*}
 I_{\{|X_t|\leq R\}}\langle X_t,b(X_t,\mu_t) \rangle 
 &= I_{\{|X_t|\leq R\}} [ \langle X_t, b(0,\delta_0) \rangle + \langle X_t,b(X_t,\mu_t)-b(0,\delta_0) \rangle] \\
 &\leq  I_{\{|X_t|\leq R\}} [\langle X_t,b(0,\delta_0)\rangle +  \kappa(|X_t|)|X_t|^2+L_1|X_t| \W_1(\mu_t,\delta_0)] \\
 & \leq \kappa_{R} R^{2} + R |b (0,\delta_0)| + L_{1}  I_{\{|X_t|\leq R\}} |X_t|  \E[|X_{t}|];   
\end{align*}
note that Assumption $\ref{b growth}$ implies $|b(0,\delta_0)|<\infty$.
Then we have \begin{align*}
\d V(X_t) & \leq  (-2\lambda  I_{\{|X_t|\geq R\}} |X_{t}|^{2} + 2 L_{1} |X_t|    \E[|X_{t}|] + C_{1} ) \d t + 2\langle X_t,\sigma(X_t)\d B_t\rangle \\
 &\le  (-2\lambda   |X_{t}|^{2} + 2 L_{1} |X_t|    \E[|X_{t}|] + C_{2} ) \d t + 2\langle X_t,\sigma(X_t)\d B_t\rangle
\end{align*} where $C_{1} = dL_{2} +\kappa_{R} R^{2} + R |b (0,\delta_0)|  < \infty $ and $C_{2} : = C_{1} + 2\lambda R^{2}$.
This establishes  \eqref{dV(X_t) bound} for $L=C_2+\lambda$.

Since $L_1\leq \frac{\lambda}{2}$, we can use  \eqref{dV(X_t) bound} and integration by parts to obtain
\begin{align*}
    \d(e^{\lambda t}V(X_t))&=\lambda e^{\lambda t}V(X_t)\d t +e^{\lambda t}\d V(X_t) \\
    & \leq e^{\lambda t}\left(L-\lambda|X_t|^2+\lambda|X_t|\E[|X_t|]\right)\d t + 2\langle X_t,\sigma(X_t)\d B_t\rangle,
\end{align*}
 which, in turn, leads to
\begin{align*}
    \E[e^{\lambda t}V(X_t)-V(X_0)]&\leq \int_0^t e^{\lambda s}\left(L-\lambda \E[|X_s|^2]+\lambda \E[X_s]^2\right)\d s\\
    &\leq \int_0^t e^{\lambda s}L\d s
     \leq \frac{L}{\lambda}e^{\lambda t}.
\end{align*}
This  gives \eqref{E[V(X_T) bound} and completes the proof.
\end{proof}

Now we construct a new function $f:\R_+ \to \R_+$ as follows.   Using the function $V(x)=1+|x|^2$ and the constants $\lambda, L$ in  Assumption \ref{drift}  and Lemma \ref{dV lemma},   we can define
\begin{align*}
    S_3 & = \{(x,y):V(x)+V(y)\leq \frac{2L}{\lambda}\}, \\
    S_4 & = \{(x,y):V(x)+V(y)\leq \frac{8L}{\lambda}\}, \\
    R_3 & = \sup_{(x,y)\in S_3}|x-y|, \\
    R_4 & = \sup_{(x,y)\in S_4}|x-y|.
\end{align*}
As before,  put $z=x-y$ and $r=|z|$. Then we define the function $f:\R_+ \to \R_+$ by
\begin{equation}
\label{e:f-defn-result2}
f(r)=\int_0^{r \wedge R_4}\varphi(s)g(s)\d s,
\end{equation}
where
\begin{align}\label{e:eta-defn-result2}
\nonumber \varphi(r)&=e^{-h(r)}, \\
\nonumber \Phi(r)& =\int_0^r\varphi(s)\d s,\\
\nonumber h(r)&=\frac{2}{D^2}\int_0^r(s\kappa^*(s)+A)\d s+\frac{8 \tilde M}{D}r, \\
\nonumber g(r)&=1-\frac{\eta}{4}\int_0^{r \wedge R_4}\Phi(s)\varphi(s)^{-1}\d s-\frac{\xi}{4}\int_0^{r\wedge R_3}\Phi(s)\varphi(s)^{-1}\d s, \\  \intertext{and }
  \eta^{-1}& =\int_0^{R_4}\Phi(s)\varphi(s)^{-1}\d s, \ \ \ \xi^{-1}=\int_0^{R_3}\Phi(s)\varphi(s)^{-1}\d s, \quad \tilde M=2M+\|\sigma(0)\|.
\end{align}  Note that $\|\sigma(x)+\sigma(y)\|\leq \tilde M$ for all $x,y$.  In the above and in the remainder of this subsection, $A$ is defined as in \eqref{A definition},    $\kappa^{*}(r) :  = \kappa(r) \vee A$,
and $D: =2\Lambda^{-1}-\sqrt{2}M$ as in \eqref{e:lam-lam-condition}.
Note $D>0$ by Assumption \ref{bound 2*}.
We further define the semi-metric
\begin{equation}
\label{e:semimetric}
\rho_1(x,y)=f(|x-y|)(1+\e V(x)+ \e V(y)), \quad x,y\in \R^{d},
\end{equation}
where $\e=\frac{\xi D^2}{16L}$.

Note that  $\varphi(r)\leq 1$ and is decreasing and   $1/2 \leq g(r) \leq 1$. Therefore, it follows that  for $r<R_4$
\begin{equation}\label{f<phi<r}
    r\varphi(R_4) \leq \Phi(r) \leq 2f(r) \leq 2\Phi(r) \leq 2r
\end{equation}
Furthermore, straightforward computations reveal that   for $r<R_4$, we have
\begin{align}
\label{f'-contraction2}    f'(r) &= \varphi(r)g(r), \intertext{and}
\nonumber    f''(r)  
    & = -h'(r)f'(r)-\left(\frac{\eta}{4}\Phi(r)\varphi(r)^{-1}+\frac{\xi}{4}\Phi(r)\varphi(r)^{-1}I_{\{r<R_3\}}\right)\varphi(r) \\
 \nonumber   &= -f'(r)\left(\frac{2}{D^2}(r\kappa^*(r)+A)+ \frac{8\tilde M}{D}\right)-\frac{\eta}{4}\Phi(r)-\frac{\xi}{4}\Phi(r)I_{\{r<R_3\}}\\
 \label{f''-contraction2}   &\leq -f'(r)\left(\frac{2}{D^2}(r\kappa(r)+A)+ \frac{8\tilde M}{D}\right)-\frac{\eta}{4}f(r)-\frac{\xi}{4}f(r)I_{\{r<R_3\}},
\end{align}
where the last inequality follows from \eqref{f<phi<r} and the fact that $\kappa(r)\leq \kappa^*(r)$.

We are now ready to prove Theorem \ref{Contraction result 2}.  The essence  of the proof is to  bound $\E[\rho_1(X_t,Y_t)]$, in which $(X_t,Y_t)$ is the  coupling process constructed in \eqref{e:coupling} in Section  \ref{sect-prelim} and $\rho_{1}$ is the semi-metric defined in \eqref{e:semimetric}.  In view of the definition of $\rho_{1}$, we will estimate the terms $f(r_{t})$, $G(X_t,Y_t)$, and their cross variation separately, where the function $f$ is defined in \eqref{e:f-defn-result2} and   $G(x, y) : = 1 + \e V(x) + \e V(y)$. These estimations together with integration by parts will then give us the desired result.


\begin{proof}[Proof of Theorem \ref{Contraction result 2}] 
As we explained earlier, the goal is to bound $\E[\rho_1(X_t,Y_t)]$, where $(X_t,Y_t)$ is the  coupling process given in \eqref{e:coupling}; 
note that an arbitrary $\delta \in (0, 1) $ was used in the construction of the coupling process. The rest of the  proof is divided into five steps.

\underline{Step 1.} {\em In this step,  we calculate a bound on $df(r_t)$.}
Using the coupling process $U_{t} = (X_{t}, Y_{t})$ and the same notation as those given in Section \ref{sect-prelim}, similar arguments as those for \eqref{df(r)} reveal that for $r_t \in (0,R_4)\setminus\{R_3\}$, we still have
\begin{align*}
   \lefteqn{ \d f(r_t)} \\ & \leq I_{\{r_{t} > 0 \}} f'(r_t)\bigg(\langle e_t, \beta_t \rangle + \frac{\rc^2(U_t)}{2r_t}\left(\tr(\alpha_t \alpha_t^\top )-|\alpha_t^\top e_t|^2\right) + \frac{\sc^2(U_t)}{2r_t}(\tr(\Delta_t\Delta_t^\top )-|\Delta_t^\top  e_t|^2)\bigg)\d t \\
    &\quad + I_{\{r_{t} > 0 \}}\frac{f''(r_t)}{2}\left(\rc^2(U_t)|\alpha_t^\top e_t|^2+\sc^2(U_t)|\Delta_t^\top e_t|^2\right)\d t \\
    &\quad +I_{\{r_{t} > 0 \}}  f'(r_t)\left(\rc(U_t) \langle e_t, \alpha_t \d B_t^2 \rangle +\sc(U_t) \langle e_t, \Delta_t \d B_t^1 \rangle\right).
\end{align*}
Since $f$ is  constant for $r\ge R_4$, we have $df(r_t)=0$  on the set  $\{r_t\ge R_4\}$.

 We now focus on the set $\{r_t <  R_4\}$.
First, the same observation for \eqref{e1:e-betas-estimate} leads to
\begin{equation*}
 I_{\{r_{t} > 0 \}}   \langle e_t, \beta_t \rangle \leq I_{\{r_t \geq \delta\}} r_t \kappa(r_t) + I_{\{r_t < \delta\}} \kappa_{\delta} \delta +L_1 \E[r_t].
\end{equation*}
This, together with \eqref{f'-contraction2},   \eqref{f''-contraction2},  and the fact that $0 < f'(r)\leq 1$,  allows us to get
\begin{align}\label{e:df(t)-rt<R2}
  \nonumber  I_{\{r_t<R_4\}}\d f(r_t)\leq&\, I_{\{r_{t} > 0 \}} \bigg[ f'(r_t)\left(I_{\{r_t<\delta\}}\kappa_{\delta}\delta + I_{\{r_t\geq \delta\}}r_t \kappa(r_t)+L_1\E[r_t]\right) \d t \\
  \nonumber  &+\frac{f'(r_t)}{2r_t} [ \sc^2(U_t) \left(\tr(\Delta_t^\top \Delta_t)-|\Delta_t^\top e_t|^2\right)   + \rc^2(U_t)  \left(\tr(\alpha_t^\top \alpha_t)-|\alpha_t^\top e_t|^2\right)]\d t  \\ 
  \nonumber  &-\frac{f'(r_t)}{2}\left(\frac{2}{D^2}(r_t\kappa(r_t)+A)+ \frac{8\tilde M}{D}\right)\left(\rc^2(U_t)|\alpha_t^\top e_t|^2+\sc^2(U_t)|\Delta_t^\top e_t|^2\right)\d t \\
 \nonumber   &-\left(\frac{\eta}{8}+\frac{\xi}{8} I_{\{r_t<R_3\}}\right) f(r_t) \left(\rc^2(U_t)|\alpha_t^\top e_t|^2+\sc^2(U_t)|\Delta_t^\top e_t|^2\right)\d t  \\
  \nonumber  &+f'(r_t)\left(\rc(U_t) \langle e_t, \alpha_t \d B _t^2 \rangle +\sc(U_t) \langle e_t, \Delta_t \d B _t^1 \rangle\right)\bigg] \\
\nonumber    \leq&\, \left(L_1f'(r_t)\E[r_t]+I_{\{0 < r_{t} < \delta \}}(\kappa_{\delta}+K_0)\delta\right)\d t  \\
\nonumber    &+I_{\{r_t\geq \delta\}}f'(r_t) \left[ r_t\kappa(r_t)+\frac{\tr(\Delta_t^\top \Delta_t)-|\Delta_t^\top e_t|^2}{2r_t}-\frac{|\alpha_t^\top e_t|^2}{D^2}\left(r_t\kappa(r_t)+A\right)\right]\d t \\
\nonumber    &-I_{\{r_t\geq \delta\}} \frac18 |\alpha_t^\top e_t|^2  f(r_t) \left( \eta + \xi I_{\{r_t<R_3\}}\right) \d t
    \\\nonumber  &-\frac{8\tilde Mf'(r_t)}{2D}I_{\{r_{t} > 0 \}} \left(\rc^2(U_t)|\alpha_t^\top e_t|^2+\sc^2(U_t)|\Delta_t^\top e_t|^2\right)\d t \\
    &+I_{\{r_{t} > 0 \}} f'(r_t)\left(\rc(U_t) \langle e_t, \alpha_t \d B _t^2 \rangle +\sc(U_t) \langle e_t, \Delta_t \d B _t^1 \rangle\right),
\end{align}
where again $\kappa_{\delta} : = \sup\{|\kappa(r)|:r \in [0,\delta]\}$. Note that   Lemma \ref{A bound} is used to derive the last inequality with $K_0= d L_{2}$    as
 in \eqref{K_0 ref}.  
 Also recall the bounds $|\alpha_t^\top e_t|^2\geq D^2$  and $f(r)\leq r$ derived in  \eqref{e:alpha e>D>0} and \eqref{f<phi<r}, respectively.  
 Then we can estimate the second and fourth terms of \eqref{e:df(t)-rt<R2} by
\begin{align}\label{e2:df(t)-rt<R2}
\nonumber \lefteqn{I_{\{0 < r_t<\delta\}}(\kappa_{\delta}+K_0)\delta  - I_{\{r_t\geq \delta\}} \frac18 |\alpha_t^\top e_t|^2  f(r_t) \left( \eta + \xi I_{\{r_t<R_3\}}\right)}\\
 \nonumber  & \leq I_{ \{0 < r_t<\delta\} }(\kappa_{\delta}+K_0)\delta -I_{ \{r_t\geq \delta\} } \frac{D^{2}}{8} f(r_t) \left( \eta+ \xi  I_{\{r_t<R_3\}}\right) \\
\nonumber   &  =  I_{ \{0 < r_t<\delta\} } \left((\kappa_{\delta}+K_0)\delta+  \frac{D^{2}}{8} f(r_t) \left( \eta+ \xi  I_{\{r_t<R_3\}}\right) \right)-   \frac{D^{2}}{8} f(r_t) \left( \eta+ \xi  I_{\{r_t<R_3\}}  \right) \\
\nonumber  &  \leq    I_{ \{0 < r_t<\delta\} } \left((\kappa_{\delta}+K_0)\delta+  \frac{D^{2}}{8} r_t \left( \eta+ \xi  I_{\{r_t<R_3\}}\right) \right)-   \frac{D^{2}}{8} f(r_t) \left( \eta+ \xi  I_{\{r_t<R_3\}}  \right)  \\
   & \leq    \left(\kappa_{\delta}+K_0+\frac{D^2(\eta +\xi)}{8}\right)\delta-  \frac{D^{2}}{8} f(r_t) \left( \eta+ \xi  I_{\{r_t<R_3\}}  \right).
\end{align}
Now, using the definition of $A$ in \eqref{A definition} and the bound \eqref{e:alpha e>D>0},  we notice the third term in \eqref{e:df(t)-rt<R2} is
\begin{equation}
\label{e3:df(t)-rt<R2}
f'(r_t)\left(r_t\kappa(r_t)+\frac{\tr(\Delta_t^\top \Delta_t)-|\Delta_t^\top e_t|^2}{2r_t}\right)-\frac{f'(r_t)|\alpha_t^\top e_t|^2}{D^2}\left(r_t\kappa(r_t)+A\right)\leq 0.
\end{equation}
Plugging \eqref{e2:df(t)-rt<R2} and \eqref{e3:df(t)-rt<R2} into \eqref{e:df(t)-rt<R2},   we have for $r_t<R_4$,
\begin{align}\label{df(rt)}
  \nonumber  \d f(r_t) \leq& \, - \frac{D^{2}}{8} f(r_t) \left( \eta+ \xi  I_{\{r_t<R_3\}}  \right) \d t  +\left(L_1f'(r_t))\E[r_t]+\delta\left(\kappa_{\delta}+K_{0}+\frac{D^2(\eta+\xi)}{8}\right)\right)\d t \\
 \nonumber   &- I_{\{r_{t} > 0 \}} \frac{4\tilde Mf'(r_t)}{D}\left(|\alpha_t^\top e_t|^2\rc^2(U_t)+|\Delta_t^\top e_t|^2\sc^2(U_t)\right)\d t \\
    &+I_{\{r_{t} > 0 \}} f'(r_t)\left(\rc^2(U_t) \langle e_t, \alpha_t \d B _t^2 \rangle +\sc^2(U_t) \langle e_t, \Delta_t \d B _t^1 \rangle\right).
\end{align}

\underline{Step 2:} {\em This step aims to derive an upper bound for $\d G(X_{t}, Y_{t})$ in \eqref{dG(X,Y)}. }  Recall that  $G(x, y) = 1 + \e V(x) + \e V(y)$.
Now we use \eqref{dV(X_t) bound} in  Lemma \ref{dV lemma} combined with the facts that $2 L_{1} \le \lambda$ and  $V(x)\geq 2 |x|$ to get
\begin{align*}
    \d V(X_t) &\leq (L-\lambda V(X_t))\d t  + \lambda V(X_t)\E[V(X_t)]\d t  + 2\langle X_t, \sigma(X_t)\d B_t \rangle, \\
    \d V(Y_t) &\leq (L-\lambda V(Y_t))\d t  + \lambda V(Y_t)\E[V(Y_t)]\d t  + 2\langle Y_t, \sigma(Y_t)\d\hat{B}_t \rangle
\end{align*}
where
\begin{align*}
    B_t&=\int_0^t \rc(U_s)\d B _s^1+\int_0^t \sc(U_s)\d B _s^2, \ \ \ \ \ \text{and}\\
    \hat{B}_t&=\int_0^t \rc(U_s)H_s \d B _s^1 + \int_0^t \sc(U_s)\d B _s^2.
\end{align*}
By the definition of $R_{1}$,  on the set $\{r_t\geq R_3  \}$, we have $V(X_t)+V(Y_t)\geq 2L/\lambda$  and hence  $2L\e -\lambda \e V(X_t)- \lambda \e V(Y_t) \leq 0$.  Similarly, on the set $\{r_t \geq R_4\}$, we have $V(X_t)+V(Y_t)\geq 8L/\lambda$ and so
\begin{align*}
    2L\e -\lambda \e V(X_t)- \lambda \e V(Y_t) &\leq 2L\e -\lambda \e \frac{4L}{\lambda}- \frac{\lambda \e}{2}V(X_t) - \frac{\lambda \e}{2}V(Y_t) \\
    &= -2L\e -\frac{\lambda \e}{2}V(X_t)-\frac{\lambda \e}{2}V(Y_t) \\
    &=-\frac{\lambda}{2}G(X_t,Y_t)+ \frac{\lambda}{2}-2L\e \\
    &\leq -\min\left\{\frac{\lambda}{2}, 2L\e \right\}G(X_t,Y_t).
\end{align*}
The last inequality follows because if $\lambda/2 \leq 2L\e$, then $ -\frac{\lambda}{2}G(X_t,Y_t)+ \frac{\lambda}{2}-2L\e \le -\frac{\lambda}{2}G(X_t,Y_t)$, and if $2L\e \leq \lambda/2$, then
\begin{align*}
    &-\frac{\lambda}{2}G(X_t,Y_t)+\frac{\lambda}{2}-2L\e = -2 L\e G(X_{t}, Y_{t}) + \bigg(\frac\lambda 2 - 2 L\e\bigg) (1-G(X_{t}, Y_{t}) ) \le -2 L\e G(X_{t}, Y_{t}).  \end{align*}
Furthermore, since $\e=\frac{\xi D^2}{16L}$ and $\eta \leq \xi$, we have $2 L\e \ge \frac{\eta D^{2}}{8}$ and therefore,
\[2\e L -\e \lambda V(X_t)- \e \lambda V(Y_t)\leq -\min\left\{\frac{\lambda}{2},\frac{\eta D^2}{8}\right\}G(X_t,Y_t).\]
Then by Lemma \ref{dV lemma},
\begin{align}\label{dG(X,Y)}
\nonumber    \d G(X_t,Y_t)=&\, \e \d V(X_t)+ \e \d V(Y_t) \\
 \nonumber   \leq& \left(2\e L -\e \lambda V(X_t)- \e \lambda V(Y_t) + 2\e L_1(V(X_t)\E[V(X_t)]+V(Y_t)\E[V(Y_t)])\right)\d t  \\
  \nonumber  &+2\e(\langle X_t,\sigma(X_t)\d B _t \rangle + \langle Y_t, \sigma(Y_t) \d\hat{B}_t \rangle) \\
    \leq& \left(I_{\{r_t <R_3\}}2L\e - I_{\{r_t \geq R_4\}}\min\left\{\frac{\lambda}{2},\frac{\eta D^2}{8}\right\}G(X_t,Y_t)\right)\d t  \\
\nonumber    &+ 2\e L_1(V(X_t)\E[V(X_t)]+V(Y_t)\E[V(Y_t)])\d t  \\
\nonumber    &+2\e(\langle X_t,\sigma(X_t)\d B _t \rangle + \langle Y_t, \sigma(Y_t) \d\hat{B}_t \rangle.
\end{align}

\underline{Step 3:} {\em We now estimate the  cross variation of $f(r_{t})$ and $G(X_{t}, Y_{t})$}.
 The cross variation of $f(r_{t})$ and $G(X_{t}, Y_{t})$ is given by 
\begin{equation}\label{d[fG]}
    \begin{aligned}
    \d[f(r),G(X,Y)]_t &= 2\e f'(r_t)\sc^2(U_t) \langle\Delta_t^{\top } e_t, \sigma(X_t)^\top X_t + \sigma(Y_t)^\top Y_t\rangle \d t  \\
    &\quad +2\e f'(r_t)\rc^2(U_t)\langle \alpha_t^{\top }e_t,\sigma(X_t)^\top X_t+(\sigma(Y_t)H_t)^\top Y_t\rangle \d t .
    \end{aligned}
\end{equation}
Furthermore,  we notice that using the Cauchy-Schwarz inequality, followed by the bound \eqref{e:sigma-nondeg} and the fact that $\|H\|=1$, we get
\begin{equation}
\begin{aligned}\label{alpha G}
   2\e &f'(r_t)\rc^2(U_t)\langle \alpha_t^{\top }e_t,\sigma(X_t)^\top X_t+(\sigma(Y_t)H_t)^\top Y_t\rangle \\
   & \leq 2\e f'(r_t)\rc^2(U_t)|\alpha_t^{\top } e_t|(|\sigma(X_t)^\top X_t|+|H_t^\top \sigma(Y_t)^\top Y_t|) \\
  & \le  2\e f'(r_t)\rc^2(U_t)|\alpha_t ^{\top }e_t| \tilde M (|X_{t}| + |Y_{t}|)\\
   & \leq \frac{\tilde M f'(r_t)}{D}\rc^2(U_t)|\alpha_t^{\top } e_t|^2G(X_t,Y_t),
\end{aligned}
\end{equation}
where we also used \eqref{e:alpha e>D>0} to derive the last inequality. 
Similarly, since $f'(r_t)\leq 1$ and $\sc^2(U_t)=0$ for $r_t\geq \delta$,   we have from Assumption \ref{bound 1} that
\begin{equation}\label{delta G}
    \begin{aligned}
     2\e&  f'(r_t)\sc^2(U_t)\langle\Delta_t^\top  e_t, \sigma(X_t)^\top X_t+\sigma(Y_t)^\top Y_t\rangle\\
  &  \leq \tilde M f'(r_t)\sc^2(U_t)|\Delta_t^\top  e_t|G(X_t,Y_t)\\
  &  \leq  \tilde M\sqrt{L_2}f'(r_t)\sc^2(U_t)r_tG(X_t,Y_t) \\
  &  \leq  \tilde M\sqrt{L_2}\delta G(X_t,Y_t).
    \end{aligned}
\end{equation}
Using \eqref{alpha G} and \eqref{delta G} in \eqref{d[fG]}, we obtain
\begin{align}
\label{e:d[fG]}
   \d[f(r),G(X,Y)]_t & \le  \frac{\tilde M f'(r_t)}{D}\rc^2(U_t)|\alpha_t^{\top } e_t|^2G(X_t,Y_t) \d t+ \tilde M\sqrt{L_2}\delta G(X_t,Y_t)\d t.
\end{align}

\underline{Step 4:} {\em In this step, we 
derive    upper  bounds for  $\d(\rho_1(X_t,Y_t))$ and $\E[e^{ct}\rho_1(X_t,Y_t)]$, in which $c$ is a suitable constant to be determined later.}
Now we can use integration by parts to combine \eqref{df(rt)}, \eqref{dG(X,Y)}, and \eqref{e:d[fG]}, 
and  recall  that $f'(r_t)=0$ for $r_t>R_4$ to get
\begin{align}\label{e:drho1}
\nonumber    \d(\rho_1(X_t,Y_t)) 
   &  = G(X_t,Y_t) \d f(r_t)+f(r_t)\d G(X_t,Y_t)+\d[f(r),G(X,Y)]_t \\
\nonumber  &  \leq  G(X_t,Y_t)\left[-\frac{\eta D^2}{8}f(r_t)I_{r_t<R_4}-\frac{\xi D^2}{8}f(r_t)I_{r_t<R_3}+L_1\E[r_t]\right]\d t  \\
\nonumber    &\quad +G(X_t,Y_t)\left[\left(\kappa_{\delta}+K_{0}+\frac{D^2(\eta+\xi)}{8}+\tilde M\sqrt{L_2}\right)\delta\right]\d t  \\
\nonumber    &\quad+G(X_t,Y_t)\left[-\frac{4\tilde Mf'(r_t)}{D}\left(|\alpha_t^\top e_t|^2\rc^2(U_t)+|\Delta_t^\top e_t|^2\sc^2(U_t)\right)\right]\d t  \\
\nonumber    &\quad+ f(r_t)\left[\frac{\xi D^2}{8}I_{\{r_t<R_3\}}-\min\left\{\frac{\lambda}{2}, \frac{\eta D^2}{8}\right\}G(X_t,Y_t)I_{\{r_t \geq R_4\}}\right]\d t  \\
\nonumber    &\quad+2f(r_t)L_1 \e(V(X_t)\E[V(X_t)]+V(Y_t)\E[V(Y_t)])\d t \\
\nonumber    &\quad+ \frac{\tilde M f'(r_t)}{D}\rc^2(U_t)|\alpha_t^{\top } e_t|^2G(X_t,Y_t) \d t+\d M_t\\
    &  \leq \left[-\min\left\{\frac{\lambda}{2},\frac{\eta D^2}{8}\right\}f(r_t)G(X_t,Y_t)+L_1\E[r_t]G(X_t,Y_t)\right]\d t \\
\nonumber    & \quad + G(X_t,Y_t) \left(\kappa_{\delta}+K_{0}+\frac{D^2(\eta + \xi)}{8}+\tilde M\sqrt{L_2}\right)\delta \d t  \\
\nonumber    & \quad +  2f(r_t)L_1 \e(V(X_t)\E[V(X_t)]+V(Y_t)\E[V(Y_t)])\d t  +\d M_t,
\end{align} where $M_t$ is a local martingale.

Next, we observe that there exists a constant $K_{1} > 0$ such that 
\begin{equation}
\label{e:r<=rho1(x,y)}
  |x-y| \leq K_1f(|x-y| ) G(x,y) = K_{1} \rho_{1}(x,y), \quad\text{ for all }x,y\in \R^{d}.
\end{equation}Indeed, when $|x-y|  < R_{2}$, \eqref{e:r<=rho1(x,y)} follows from \eqref{f<phi<r} and the fact that $G(x,y) \ge 1$. When  $|x-y| \ge R_{2}$, we have $f(|x-y| ) = f(R_{2})$ and hence \begin{displaymath}
|x-y| \le |x| + |y| \le \frac12 (V(x) + V(y)) \le \frac{1}{2\e f(R_{2})} f(|x-y| ) G(x,y).
\end{displaymath}
In view of \eqref{e:r<=rho1(x,y)}, we have the bound
\begin{equation}
\label{e2:drho1}
L_1 \E[r_t]G(X_t,Y_t)\leq L_1K_1 \E[\rho_1(X_t,Y_t)]G(X_t,Y_t).
\end{equation}
Furthermore, we have
\[\e \E[V(X_t)]V(X_t)+\e \E[V(Y_t)]V(Y_t)\leq \e^{-1} G(X_t,Y_t) \E[G(X_t,Y_t)].\]
Therefore
\begin{equation}
\label{e3:drho1}
2L_1 \e f(r_t)(\E[V(X_t)]V(X_t)+\E[V(Y_t)]V(Y_t))\leq \frac{2L_1}{\e}f(r_t)G(X_t,Y_t) \E[G(X_t,Y_t)].
\end{equation}
Now with $ c=\frac12 \min\{\lambda/2,\eta D^2/8\}$ and using integration by parts and  \eqref{e:drho1}, \eqref{e2:drho1},  and \eqref{e3:drho1}, we get
\begin{equation}
\label{e:E[e^ct rho1]}
\E[e^{ct}\rho_1(X_t,Y_t)]\leq \rho_1(X_0,Y_0)+\int_0^te^{cs}\E[\mathcal{J}_s]\d s,
\end{equation}
where
\begin{equation}\label{J_s}
\begin{aligned}
\mathcal{J}_s=&-c\rho_1(X_s,Y_s)+L_1K_1\E[\rho_1(X_s,Y_s)]G(X_s,Y_s)\\ &+\left(\kappa_{\delta}+K_{0}+\frac{D^2(\eta+\xi)}{8}+\tilde M\sqrt{L_2}\right)\delta G(X_s,Y_s)+\frac{2L_1}{\e}\E[G(X_s,Y_s)]\rho_1(X_s,Y_s).
\end{aligned}
\end{equation}

We first consider the third term in \eqref{J_s}. From Lemma
 \ref{dV lemma} and the fact that $c<\lambda$, we have
\begin{align*}
    \int_0^te^{cs}\E[G(X_s,Y_s)]\d s &=\int_0^te^{cs}\left(1+\e \E[V(X_s)]+\e \E[V(Y_s)]\right)\d s \\
    &\leq \int_0^te^{cs}\left(1+\frac{2L}{\lambda}+\e e^{-\lambda s}(\E[V(X_0)]+\E[V(Y_0)])\right)\d s \\
    &\leq \frac{1}{c}\left(1+\frac{2L}{\lambda}\right)e^{ct} + \frac{\e}{\lambda-c}(\E[V(X_0)+\E[V(Y_0)]).
\end{align*}
Consequently,  there exists a $K_2>0$ not depending on $\delta$ such that
\begin{align*}\left(\kappa_{\delta}+K_0+\frac{D^2(\eta + \xi)}{8}+\tilde M\sqrt{L_2}\right)\delta\int_0^te^{cs}\E[G(X_s,Y_s)]\d s
\\ \leq K_2\delta e^{ct}+ K_{2}\e (\E[V(X_0)+\E[V(Y_0)])\delta.\end{align*} 
Now by   Lemma
 \ref{dV lemma}, we have
\begin{equation}\label{E[G(X_t,Y_t)]}
\E[G(X_t,Y_t)]\leq 1+\frac{2L\e}{\lambda}+e^{-\lambda t}\e(\E[V(X_0)]+\E[V(Y_0)]).
\end{equation}
Next we estimate the second and fourth terms in \eqref{J_s}.  Define $K_3=L_1K_1+2L_1/\e$, $K_4=1+2L\e/\lambda$, and $K_5=K_5(X_0,Y_0)=\e(\E[V(X_0)]+\E[V(Y_0)])$.  Then we get
\begin{align*}
    L_1&K_1\int_0^te^{cs}\E[\rho_1(X_s,Y_s)]\E[G(X_s,Y_s)]\d s +\frac{2L_{1}}{\e}\int_0^te^{cs}\E[G(X_s,Y_s)]\E(\rho_1(X_s,Y_s)]\d s \\
    & = K_3\int_0^te^{cs}\E[G(X_s,Y_s)]\E[\rho_1(X_s,Y_s)]\d s \\
    & \leq K_3K_4\int_0^te^{cs}\E[\rho_1(X_s,Y_s)]\d s + K_3K_5\int_0^te^{(c-\lambda)s}\E[\rho_1(X_s,Y_s)]\d s .
\end{align*}

Then we have
\begin{align*}
    \int_0^te^{cs}\E[\mathcal{J}_s]\d s\leq& \, (K_2 e^{ct}+K_2 K_{5})\delta+K_3K_4\int_0^te^{cs}\E[\rho_1(X_s,Y_s)]\d s \\
    &+K_3K_5\int_0^te^{(c-\lambda)s}\E[\rho_1(X_s,Y_s)]\d s -c\int_0^te^{cs}\E[\rho_1(X_s,Y_s)]\d s .
\end{align*}
 Plugging this into  \eqref{e:E[e^ct rho1]} gives us
\begin{align*}\E[e^{ct}\rho_1(X_t,Y_t)]& \leq \E[\rho_1(X_0,Y_0)]+(K_2 e^{ct}+K_{2}K_5)\delta\\ & \quad +\left(K_3(K_4+K_5)-c\right)\int_0^te^{cs}\E[\rho_1(X_s,Y_s)]\d s.\end{align*}

\underline{Step 5:} {\em In this final step, we establish \eqref{contraction 1}, \eqref{contraction 2}, and \eqref{contraction 3}.}  First we recall that $c$ is independent of the coupling, and thus does not depend on $\delta$.
We define $$L_1^*=(K_1+2/\e)^{-1}(K_4+K_5)^{-1}c.$$   If $L_1<L_1^*$, then   $K_3(K_4+K_5)\leq c$ and hence
\[\E[\rho_1(X_t,Y_t)]\leq ( K_2 +K_2 K_{5} e^{-ct}) \delta+e^{-ct}\E[\rho_1(X_0,Y_0)].\]
Letting $\delta \to 0$,  we get \eqref{contraction 1}.
Furthermore, since $r_t\leq K_1\rho_1(X_1,Y_1)$,  \eqref{contraction 2} follows immediately from \eqref{contraction 1}.

Note
that $K_5$ depends on the initial probability measures $\mu_0,\nu_0$, so we only get a local result.  To get \eqref{contraction 3} we define $L_1^{**}=(K_1+2/\e)^{-1}K_4^{-1}c$.  Then, if $L_1<L_1^{**}$ so that $K_3K_4\leq c$, we have
\[\E[\rho_1(X_t,Y_t)]\leq e^{-ct}\W_{\rho_1}(\mu_0,\nu_0)+(K_2+K_{2}K_5e^{-ct})\delta+e^{-ct}K_3K_5\int_0^te^{(c-\lambda)s}\E[\rho_1(X_s,Y_s)]\d s.\]
Note that \eqref{f<phi<r} implies that $f(r)\leq f(R_4) \le R_{2}$ for all $r\ge 0$.  
Then, by equation \eqref{E[G(X_t,Y_t)]}, we have
\[\int_0^te^{(c-\lambda)s}\E[\rho_1(X_s,Y_s)]\d s\leq R_4\left(1+\frac{2L\e}{\lambda}+\e(\E[V(X_0)]+\E[V(Y_0)])\right)\int_0^te^{(c-\lambda)s}\d s.\]
Next, since $c=\frac{1}{2}\min\{\frac{\lambda}{2},\frac{D^2\eta}{8}\}<\lambda$ and $V(x)=1+|x|^2$, there exists a constant $K_6>0$ such that
\begin{align*}
\E[\rho_1(X_t,Y_t)] & \leq e^{-ct}\W_{\rho_1}(\mu_0,\nu_0)+e^{-ct}K_6(\e \E[V(X_0)]+ \e \E[V(Y_0)])^2\\
&\ \ +(K_2+K_{2}\e(\E[V(X_0)]+\E[V(Y_0)])e^{-ct})\delta.
\end{align*}
Letting $\delta \to 0$ gives \eqref{contraction 3}.
\end{proof}

\section{Applications}\label{sect-applications}
In this section, we explore the applications of the contraction results presented in Section \ref{sec-main}. We first show that under the conditions of Theorem \ref{Contraction result}, a propagation of chaos type result holds. Next, we show that under the conditions of Theorem \ref{Contraction result} or \ref{Contraction result 2}, \eqref{SDE 1} possesses  a unique invariant measure $\mu$ and that starting from any initial distribution $  \nu_{0}$, the distribution of $X_{t}$ converges to $\mu$ at an exponential rate.

We need the following lemma,
which  shows that the solution to \eqref{SDE 1} with initial condition $X_{0} =0$ is  $L^{1}$-bounded under certain conditions. The proof of the lemma is deferred to the Appendix in order not to disrupt the flow of presentation.

\begin{lemma}\label{nonexplosive lemma}
Suppose Assumptions {\rm\ref{b lipschitz}}, {\rm\ref{sigma lipschitz}},   {\rm\ref{b growth}}, and {\rm\ref{bound 1}} hold. Let $X_t(0)$ be the solution to \eqref{SDE 1} starting at $X_0=0$. If there exist  $R_0>0$ and $K>0$ such that
\begin{equation}\label{nonexplosive condition}
L_1+\kappa(r)<-K<0, \ \ \ \forall r>R_0,
\end{equation} then
we have
\[\sup_{t\geq 0}\E[|X_t(0)|]<\infty.\]
\end{lemma}

\subsection{Propagation of Chaos}

We now show that under the same assumptions as those in Theorem \ref{Contraction result}, a propagation of chaos type result holds. To this end, let $(Y_{0}^{i}, W^{i} )_{i\ge 1}$ be independent copies of $(X_{0}, W)$, in which $X_{0}$ is the initial condition of \eqref{SDE 1} with $\LL_{X_{0}} = \mu_{0}$ and $W$ is a standard $d$-dimensional Brownian motion.  
For each $i =1,2,\dots, n$, we define the $n$-particle system by
\begin{equation}\label{n particle system 2}
    \d X_t^{n,i}=b(X_t^{n,i},\mu_t^n)\d t+\sigma(X_t^{n,i})\d W_t^i,  \quad \ \ \ \mu_t^n=\frac{1}{n}\sum_{i=1}^n\delta_{X_t^{i,n}},
\end{equation}
where  
 $\delta_{X_t^{n,i}}$ is the Dirac measure and  so $\mu_t^n$ is the empirical mean.  We similarly define $n$-copies of the Mckean-Vlasov SDE \eqref{SDE 1} by
\begin{equation}\label{n copied McKean SDE}
\d Y_t^i=b(Y_t^i,\mu_t)\d t+\sigma(Y_t)\d W_t^i, \ \ \ \ \mu_t=\mathcal{L}(Y_t^i).
\end{equation}
For each $i$, we can define a coupling for $X_t^{n,i}$ of \eqref{n particle system 2} and $Y_{t}^{i}$  of  \eqref{n copied McKean SDE} similarly to that constructed in    Section \ref{sect-prelim}.
In addition,  we define $Z_t^i=X_t^{i,n}-Y_t^i$, and $r_t^i$, $e_t^i$, $U_t^i$, $\Delta_t^i$, $\beta_t^i$, and $\alpha_t^i$ in the same way as    in Section \ref{sect-prelim}.

\begin{proposition}\label{prop-chaos}
   Suppose  the conditions of 
   Theorem {\rm\ref{Contraction result}} and  {\rm\eqref{nonexplosive condition}}  hold. If the constant $\gamma$ given in  \eqref{e:gamma-defn} is positive,  
   then  the $n$-particle system \eqref{n particle system 2} converges to \eqref{n copied McKean SDE} in the following sense:
    \[\W_\rho(\mu_t^n,\mu_t)\leq e^{-\gamma t}\W_\rho(\mu_0^n,\mu_0)+\frac{n(d)^{1/2}}{\gamma},\]
    where
    \begin{equation}\label{n(d)}
  n(d)=C\begin{cases}
  n^{-1/2} \ \ \ \ \ &  \text{ if } \ \ d<4, \\
  n^{-1/2}\log n \ \ & \text{ if } \ \ d=4, \\
  n^{-2/d} \ \ \ \ \ &  \text{ if } \ \ d>4,  \end{cases}
  \end{equation}
  with the positive constant $C$ depends on the dimension $d$ and the moment.
 \end{proposition}

\begin{proof}  First, we note that the process $Y_t^i$ is non-explosive for any initial value $Y_0^{i}$ with $\LL(Y^{i}_{0}) = \mu_{0} \in \mathcal P_{1}(\R^{d})$.  Indeed using  Lemma \ref{linear bounds} and \eqref{e:f(r_t)estimate},
 we have
\[\E[|Y_t^{i}-Y_t(0)|]\leq \frac{D^2}{2\varphi(R_1)}\E[f(|Y_t(0)-Y_t^{i})|]\leq \frac{D^2}{2\varphi(R_1)}e^{-\gamma t}\W_\rho (\delta_0,\mu_0),\] where $Y_t(0)$ denotes the solution to \eqref{SDE 1} with initial condition $Y_{0} =0$.
Thus, we have
\[\sup_{t\geq 0}\E[|Y_t(0)-Y_t^{i}]\leq \frac{D^2}{2\varphi(R_1)}\W_{\rho}(\delta_0,\mu_0) \le  \frac{D^2}{2\varphi(R_1)} \mu_{0}(|\cdot|) <\infty.\]
This, together with Lemma \ref{nonexplosive lemma}, implies that 
\begin{equation}\label{cor eq 1}
    \sup_{t\geq 0}\E[|Y_t^{i}|]< \infty, \quad \forall i.
\end{equation}

    From Lemma \ref{radial lemma}, we have
    \begin{align*}
        \d r_t^i & =  I_{\{r_t^i> 0\}}\bigg[ \frac{1}{2r_t^i} \big[ \rc^2(U_t^i)  (\tr(\alpha_t^i\alpha_t^{i\top })-|\alpha_t^{i\top }e_t^i|^2 )+  \sc^{2}(U_t^i)  (\tr(\Delta_t^i\Delta_t^{i\top })-|\Delta_t^{i\top }e_t^i|^2 )\big]\d t \\
        &\quad +   \langle e_t^i,\beta_t^i\rangle\d t + \sc(U_t^i)\langle e_t^i,\Delta_t\d B_t^{i,1}\rangle+\rc(U_t^i)\langle e_t^i,\alpha_t^i \d B_t^{i,2}\rangle \bigg].
    \end{align*}
    Furthermore, we can define $\mathcal{E}_s^i$ as in \eqref{E} to get
    \[\E[e^{\gamma t}f(r_t^i)-f(r_0^i)]\leq \int_0^t e^{\gamma s}\E[\mathcal{E}_s^i]\d s. \]
 Next, with 
    $\nu_t^n:=\frac{1}{n}\sum_{k=1}^n\delta_{Y_t^k}, $
    we have the inequality
    \begin{align*}
        \W_1(\mu_t^n,\mu_t)&\leq \W_1(\mu_t,\nu_t^n)+\W_1(\nu_t^n,\mu_t^n) \leq W_1(\mu_t,\nu_t^n)+\E[r_t^i].
    \end{align*}
    Substituting this into equation \eqref{e1:e-betas-estimate}, we get
    \begin{align*}
        I_{\{r_{s}^i > 0 \}} \langle e_s^i, \beta^{i}_s \rangle &  \leq I_{\{r_{s}^i > 0 \}}( r_s^i\kappa(r_s^i)+L_1\mathcal{W}_1(\mu_s^n,\mu_s))\\ & \leq I_{\{r_s^i\geq \delta\}}r_s^i\kappa(r_s^i)+I_{\{r_s^i<\delta\}}\kappa_{\delta}\delta+\frac{L_1D^2}{2\varphi(R_1)}\E[f(r_s^i)]+L_1\W_1(\mu_s,\nu_s^n),
    \end{align*}
    Thus, with the above inequality,    \eqref{E bound}  becomes
    \begin{align*}
        \mathcal{E}_s^{i} \leq & \ \frac{L_1D^2}{2\varphi(R_1)}\left(\E[f(r_s^i)]-f(r_s^i)\right)+I_{\{r_s^i<\delta\}}(c+\kappa_{\delta}+dL_2)\delta \\
        &+I_{\{r_s^i\geq \delta\}}\left(cf(r_s^i)+f'(r_s^i)\left(r_s^i\kappa(r_s^i)+\frac{A}{2}\right)+\frac{D}{2}f''(r_s^i)\right) +L_1\W_1(\mu_t,\nu_t^i).
    \end{align*}
  Therefore, following the remainder of the proof for Theorem \ref{Contraction result}, we get
  \begin{equation}\label{Cor eq 1}
  \E[e^{\gamma t}f(r_t^i)]\leq \E[f(r_0^i)]+\int_0^t e^{\gamma s}\E[\W_1(\mu_t,\nu_t^i)]\d s+(c+\kappa_{\delta}+dL_{2})\delta \frac{e^{\gamma t}}{\gamma}.
  \end{equation}
  Using H\"older's inequality and \cite[Theorem 5.8]{CarmoD-18I}, we obtain the bound
  \begin{equation*}
      \E[\W_1(\mu_t,\nu_t^i)] \leq \E[\W_2(\mu_t,\nu_t^i)] \leq \left(\E[\W_2^2(\mu_t,\nu_t^i)]\right)^{1/2}\leq n(d)^{1/2}.
  \end{equation*}
    Thus we have
  \[\int_0^te^{\gamma s}\E[\W_1(\mu_t,\nu_t^i)]\d s \leq \frac{e^{\gamma t}}{\gamma}n(d)^{1/2}, \]
  which combined with \eqref{Cor eq 1} yields
    \[\W_\rho(\mu_t^n,\mu_t)\leq e^{-\gamma t}\W_\rho(\mu_0^n,\mu_0)+\frac{n(d)^{1/2}}{\gamma}+ \frac{(c+\kappa_{\delta}+dL_{2})\delta}{\gamma}.\]
    Letting $\delta\to 0$ completes the proof.
\end{proof}

\begin{proposition}\label{prop-chaos 2}
    Suppose the conditions of Theorem \ref{contraction 2} and \eqref{nonexplosive condition} hold.  If the constant $c$ given in \eqref{contraction 2} is positive, then the $n$-particle system \eqref{n particle system 2} converges to \eqref{n copied McKean SDE} in the following sense:
    \begin{equation*}
        \W_{\rho_1}(\mu_t^n,\mu_t)\leq C(e^{-ct}\W_{\rho_1}(\mu_0^n,\mu_0)+\frac{n(d)^{1/2}}{c}),
    \end{equation*}
    where $n(d)$ is defined as in \eqref{n(d)} and some constant $C$.
\end{proposition}

\begin{proof}
The result    can be deduced using a proof that mirrors closely the one for Proposition \ref{prop-chaos}, and defining instead $\mathcal{J}_s^i$ as in \eqref{J_s} and following the proof of Theorem \ref{contraction 2}.  We note however that 
\end{proof}

\subsection{Exponential Ergodicity} 
As an application of the contraction results, we demonstrate that the McKean-Vlasov process given by \eqref{SDE 1} is exponentially ergodic under the conditions of Theorems \ref{Contraction result} or \ref{Contraction result 2}. Let   $(X_{t})_{t\geq 0}$ denote the solution to \eqref{SDE 1} with initial condition $X_{0}$. Suppose  $\mathcal{L}(X_{0})=\mu$ and denote 
$P^*_{t}\mu=\mathcal{L}(X_{t})$ for $t\ge 0$. We have the semigroup property $P^{*}_{t+s} = P^{*}_{t} P^{*}_{s} $ for any $s, t\ge 0$. As usual, $\mu \in \mathcal P(\R^{d})$ is said to be an invariant measure for the semigroup $P_{\cdot}^{*}$ if  $P^{*}_{t}  \mu  = \mu$ for any $t\ge 0$.  


\begin{theorem}\label{invariant measure}
    Suppose that \eqref{nonexplosive condition}  and   the  conditions of  Theorem {\rm\ref{Contraction result}} hold with   $\gamma>0$.  
    Then $P_t^{*}$ admits an invariant measure $\mu \in \mathcal{P}_1(\R^{d})$ satisfying
    \begin{equation}\label{contraction to invariant measure}
        \mathcal{W}_1(P_t^{*}\nu_0,\mu) \leq C\mathcal{W}_1(\mu,\nu_0)e^{-\gamma t},\ \ \forall \nu_{0}\in \mathcal{P}_1(\R^{d}),
    \end{equation}
    where the constants $C$ and $\gamma$ are defined in Theorem {\rm\ref{Contraction result}}.
\end{theorem}

\begin{proof}
We adopt the shift coupling idea from the proof of \cite[Theorem 3.1]{Wang-18}.  Let $\delta_0$ be the Dirac measure at $0$.  Then $P_t^{*}\delta_0=\mathcal{L}(X_t)$.
Next we define for any $s\geq 0$,  $(\Bar{X}_t := X_{t+s}(0))_{t\geq0}$.  Note that  $\Bar{X}_t$ solves
\[\d\Bar{X}_t=b(\Bar{X}_t,\mathcal{L}(\bar{X}_t)\d t+\sigma(\Bar{X}_t)\d\Bar{W}_t, \ \ \Bar{X}_0=X_s(0)\]
for the $d$-dimensional Brownian motion $\Bar{W}_t=W_{t+s}-W_s$.  Since \eqref{SDE 1} has a unique solution, we have
\begin{equation}\label{Transition of Dirac measure}
P^*_t(P^*_s\delta_0)=\mathcal{L}(\Bar{X}_t)=\mathcal{L}(X_{t+s}(0))=P^*_{t+s}\delta_0, \ \ \forall s,t\geq 0.
\end{equation}
Now Theorem \ref{Contraction result} and equation \eqref{Transition of Dirac measure} combined with $\mathcal{L}(\Bar{X}_0)=P^*_s\delta_0$ gives us
\begin{align*}
    \mathcal{W}_1(P^*_{t+s}\delta_0,P^*_t\delta_0)&=\mathcal{W}_1(\mathcal{L}(X_t(P^*_s\delta_0)),\mathcal{L}(X_t(0))) \\
    &\leq C\mathcal{W}_1(P^*_s\delta_0,\delta_0)e^{-\gamma t} \\
    &=Ce^{-\gamma t}\E[|X_s(0)|], \ \ \ s,t\geq0
\end{align*}
Thanks to Lemma \ref{nonexplosive lemma}, $\sup_{s\geq0}\E[|X_s(0)|]<\infty$. Therefore, we have
\[\lim_{t\to\infty}\sup_{s\geq 0}\mathcal{W}_1(P^*_t\delta_0,P^*_{t+s}\delta_0)=0.\]
Consequently,  $\{P^{*}_t\delta_0\}_{t\geq 0}$ is Cauchy in $\mathcal{W}_1$. Thus, there exists some $\mu\in \mathcal{P}_1(\R^{d})$ such that 
\begin{equation} \label{cauchy}
    \lim_{t\to\infty}\mathcal{W}_1(P^*_t\delta_0,\mu)=0.
\end{equation}
Now, we combine \eqref{cauchy} with Theorem \ref{Contraction result} to derive
\[\lim_{t\to \infty}\mathcal{W}_1(P^*_s\mu,P^*_sP^*_t\delta_0)\le \lim_{t\to \infty}Ce^{-\gamma s}\mathcal{W}_1(\mu, P_{t}^{*} \delta_{0})= 0, \ \ s\geq 0.\]
This, in turn,   gives us
\begin{align*}
    \mathcal{W}_1(P_s^{*}\mu,\mu) &\leq \lim_{t\to\infty}(\mathcal{W}_1(P^*_s\mu,P^*_sP^*_t\delta_0)+\mathcal{W}_1(P^*_sP^*_t\delta_0,P^*_t\delta_0)+\mathcal{W}_1(P^*_t\delta_0,\mu))\\
    &\leq \lim_{t\to\infty}\mathcal{W}_1(P^*_sP^*_t\delta_0,P^*_t\delta_0)=\lim_{t\to\infty}\mathcal{W}_1(P^*_{t+s}\delta_0,P^*_t\delta_0)=0.
\end{align*}
As a result, $\mu$ is an invariant probability measure.  Furthermore, by Theorem \ref{Contraction result}, we have
\[\mathcal{W}_1(P^*_t\nu_0,\mu)=\mathcal{W}_1(P^*_t\nu_0,P^*_t\mu)\leq Ce^{-\gamma t}\mathcal{W}_1(\nu_0,\mu), \ t\geq 0;\] this establishes \eqref{contraction to invariant measure} and 
completes the proof.
\end{proof}

To proceed,  we  use 
 Theorem \ref{Contraction result 2} to derive exponential ergodicity for the semigroup $P_{\cdot}^{*}$ of \eqref{SDE 1}. 

\begin{theorem}\label{contraction to invariant measure 2}
Given the assumptions of Theorem {\rm\ref{Contraction result 2}}, 
$P^{*}_t$ admits an invariant measure $\mu \in \mathcal{P}_1(\R^{d})$ such that
\begin{equation}
    \W_1(P_t^*\nu_0,\mu)\leq K\W_{\rho_1}(\mu,\nu_0)e^{-ct}, \ \ \forall \nu_{0}\in \mathcal{P}_1(\R^{d}),
\end{equation}
where $K$, $c$, and $\rho_1$ are defined in Theorem {\rm\ref{Contraction result 2}}.
\end{theorem}

\begin{proof}  
We first note that Lemma \ref{dV lemma} says that $\sup_{t \ge 0} \E[V(X_{t}(0)) ] < \infty$, where $X_{t}(0)$ denotes the solution to \eqref{SDE 1} with initial condition $X_{0} = 0$. This, together with the fact that $f(x) \le f(R_{2})$ with the function $f$  defined in \eqref{e:f-defn-result2} implies that $$\sup_{t \ge 0} \E[\rho_{1}(X_{t}(0), 0) ]  \le\sup_{t \ge 0} \E[f(|X_t(0)|)(1+\e V(X_t(0))+\e V(0)) ] < \infty.$$
Next, we use \eqref{Transition of Dirac measure} 
and  \eqref{contraction 2} to obtain  
\begin{equation}\label{contraction 2 eq 1}
\begin{aligned}
\lim_{t\to \infty}\sup_{s\geq 0} \W_1(P_{t+s}^*\delta_0,P_t^*\delta_0)&=\lim_{t\to \infty}\sup_{s\geq 0} \W_1(\LL(X_t(P_s\delta_0)),\LL(X_t(0)))\\
&\leq \lim_{t\to \infty}\sup_{s\geq 0}Ke^{-ct}\W_{\rho_1}(P_s^*\delta_0,\delta_0)\\
&=\lim_{t\to \infty}\sup_{s\geq 0} K\E[\rho_1(X_s(0),0)]e^{-ct} \\
& = 0.
\end{aligned}
\end{equation} The rest of the argument is similar to that in proof of Theorem \ref{invariant measure}.
\end{proof}

\appendix
\section{Proofs of Several Lemmas}\label{sect-appendix}
\begin{proof}[Proof of Lemma \ref{radial lemma}]
We use  the idea  in the proof of  \cite[Lemma 2]{Zimmer-17}.  First we consider the map $Z_t \to |Z_t|^2=r_t^2$.  Then by It\^o's formula, we have
\begin{equation}\label{d|Z|^2}
\begin{aligned}
    \d(|Z_t|^2)=&(2\langle Z_t, \beta_t \rangle +\rc^{2}(U_t)\tr(\alpha_t\alpha_t^\top ) + \sc^{2}(U_t)\tr(\Delta_t\Delta_t^\top ))\d t \\
    &+2\sc(U_t)\langle Z_t, \Delta_t \d B _t^1 \rangle + 2\rc(U_t) \langle Z_t, \alpha_t \d B _t^2 \rangle
\end{aligned}
\end{equation}
First note that by \cite[Theorem 2.1]{Wang-18}, there exists a function $F(t)$ such that $\sup_{s\in[0,T]}\E[|Z_s|]\leq F(T)$   for all $T>0$.
Now for any $\e \in( 0, 1)$, we define the $C^2$ function $g_\e (r): [0,\infty) \to (0,\infty)$ by
\begin{equation}\label{g epsilon}
g_\e(r)=
    \begin{cases}
    -\frac{r^2}{8\e^{3/2}} + \frac{3r}{4\e^{1/2}} + \frac{3\e^{1/2}}{8}, & r< \e \\
    \sqrt{r},  & r\geq \e.
    \end{cases}
\end{equation} 
Then we have
\begin{equation*}
\frac{dg_\e (r)}{dr}=
    \begin{cases}
    -\frac{r}{4\e^{3/2}}+\frac{3}{4\e^{1/2}}, & r < \e,\\
    \frac{1}{2\sqrt{r}}, & r \geq \e,
    \end{cases}
\quad \text{ and }\quad
    \frac{d^2g_\e(r)}{dr^2}=
    \begin{cases}
    -\frac{1}{4\e^{3/2}}, & r<\e, \\
  -  \frac{1}{4r^{3/2}}, & r \geq \e.
    \end{cases}
\end{equation*}
Clearly 
 $g_\e(r) \in C^2(\mathbb{R_+})$.  Now fix $\delta>0$ and let $\e^{2}<\delta/2$, then we have by \eqref{d|Z|^2} and It\^o's formula
\begin{align}\label{g epsilon ito}
     \nonumber   g_\e&(|Z_t|^{2})-g_\e(|Z_0|^{2}) \\   \nonumber    =&\,  \!\int_0^tI_{\{|Z_s|^{2}\geq \e\}}\!\bigg(\langle e_s,\beta_s \rangle + \frac{\sc^2(U_s)}{2r_s}(\tr(\Delta_s \Delta_s^\top ) -|\Delta_s^\top  e_s|^2)  + \frac{\rc^2(U_s)}{2r_s}(\tr(\alpha_s \alpha_s^\top ) -|\alpha_s^\top  e_s|^2)  \bigg)\d s  \\
   \nonumber      &+    \int_0^t I_{\{|Z_s|^{2}\geq \e\}} \sc(U_s)\langle e_s,\Delta_s \d B _s^1 \rangle + \int_0^t \rc(U_s) \langle e_s, \alpha_s\d B _s^2 \rangle \\
   \nonumber      &+   \int_0^t I_{\{|Z_s|^{2}<  \e\}} \left( \left(-\frac{r_{s}^{2}}{4\e^{3/2}}+\frac{3}{4\e^{1/2}}\right)\left(\langle Z_s,\beta_s \rangle + \tr(\Delta_s \Delta_s^\top )\right)-\frac{1}{2\e^{3/2}}|\Delta_s^{\top } Z_s|^2\right)\d s   \\
    &+  \int_0^t I_{\{|Z_s|^{2}<  \e\}} \left(-\frac{r_{s}^{2}}{4\e^{3/2}}+\frac{3}{4\e^{1/2}}\right)\langle Z_s, \Delta_s \d B _s^1 \rangle.
    \end{align}
Note that in the above, we used the facts that $\rc(U_t)=0$ and $\sc(U_t)=1$ when $r_t< \delta/2$.

We will  derive \eqref{radial process} by passing to the limit when $\e \to 0$ in \eqref{g epsilon ito}. To this end,
note that
\begin{equation}
\label{e1lem3.1proof}
\frac{1}{2\e^{1/2}}\leq -\frac{r}{4\e^{3/2}}+\frac{3}{4\e^{1/2}} \leq \frac{3}{4\e^{1/2}}
\end{equation} for  $0 \le  r \le \e$.
 Using Assumption \ref{b lipschitz} along with the facts that the function $\kappa$ is  
  bounded 
   and that $\mathcal{W}_1(\mu_{t},\nu_{t}) \le \E[|Z_{t}|] \le F(T) $ for any $t\in [0, T]$,  
   we get
\begin{equation}
\label{e2lem3.1proof}
I_{\{|Z_t|^{2}<\e\}} |\langle Z_t,\beta_t \rangle| \leq I_{\{|Z_t|^{2}<\e\}}(\kappa(|Z_{t}|)|Z_t|^2+L_1\mathcal{W}_1(\mu_t,\nu_t)|Z_t|)\leq \kappa_{0}\e+L_{1 }F(T)\e^{1/2}.
\end{equation} 
Furthermore, Assumption \ref{sigma lipschitz} along with the fact that $\|A\|^2\leq \tr(AA^\top )\leq d\|A\|^2$ for any  $A \in \R^{d \times d}$   shows that   on the set  $\{|Z_t|^{2}<\e \}$
\begin{equation}
\label{e3lem3.1proof}
\tr(\Delta_t\Delta_t^\top )\leq d\|\Delta_t\|^2\leq dL_2|Z_t|^2\leq dL_2\e.
\end{equation}
Similarly, Assumptions \ref{b lipschitz} and \ref{sigma lipschitz} imply that on the set $\{|Z_t|^{2}<\e \}$, we have
\begin{equation}
\label{e4lem3.1proof}
|\Delta_t^{\top } Z_t|^2\leq|Z_t|^2 \|\Delta_t\|^2\leq L_2\e^2.
\end{equation}
Combining  equations \eqref{e1lem3.1proof}--\eqref{e4lem3.1proof}, we derive
\begin{displaymath}
I_{\{|Z_s|^{2}<\e\}}   \left| \bigg(-\frac{r_{s}^{2}}{4\e^{3/2}}+\frac{3}{4\e^{1/2}}\bigg)\left(\langle Z_s,\beta_s \rangle + \tr(\Delta_s \Delta_s^\top )\right)-\frac{1}{4\e^{3/2}}|\Delta_s^{\top } Z_s|^2 \right| \le K \e^{1/2} \le K.
\end{displaymath} for some positive constant $K$ independent of $\e$ and $s\in [0, t]$.
As a result, it follows from   the dominated convergence theorem that as $\e \to 0$
\begin{equation}
\label{e5lem3.1proof}
\int_0^t I_{\{|Z_s|^{2}<\e\}} \bigg( \bigg(-\frac{r_{s}^{2}}{4\e^{3/2}}+\frac{3}{4\e^{1/2}}\bigg)\left(\langle Z_s,\beta_s \rangle + \tr(\Delta_s \Delta_s^\top )\right)-\frac{1}{4\e^{3/2}}|\Delta_s^{\top }Z_s |^2\bigg)\d s\to 0, \ \ \ \text{a.s.}
\end{equation}
Equations  \eqref{e1lem3.1proof} and \eqref{e4lem3.1proof} also enable us to derive
\begin{align*}
\E\bigg[ \int_0^t I_{\{|Z_s|^{2}<\e\}}  \left(-\frac{r_{s}^{2}}{4\e^{3/2}}+\frac{3}{4\e^{1/2}}\right)^{2} | \Delta_s^{\top }  Z_s|^{2}\d s\bigg] \le K t\e  \to 0
\end{align*}
as $\e \to 0$.
Therefore there exists a subsequence $\{\e_k\}$  where $\e_k \to 0$ as $k\to \infty$ such that
\begin{equation}
\label{e6lem3.1proof}\int_0^t I_{\{|Z_s|<\e\}} \left(-\frac{r}{4\e^{3/2}}+\frac{3}{4\e^{1/2}}\right)\langle Z_s, \Delta_s \d B_s^1 \rangle   \to 0\quad  \text{ a.s. }
\end{equation}
as $\e_k \to 0$.

Using similar  arguments as those for \eqref{e5lem3.1proof} and  \eqref{e6lem3.1proof}, we can show that along a
subsequence of $\{\e_{k}\}$ (still denoted by  $\{\e_{k}\}$), we have
\begin{align*}
&  \!\int_0^t I_{\{|Z_s|^{2}\geq \e\}} \!\bigg(\langle e_s,\beta_s \rangle + \frac{\sc^2(U_s)}{2r_s}(\tr(\Delta_s \Delta_s^\top ) -|\Delta_s^\top  e_s|^2)  + \frac{\rc^2(U_s)}{2r_s}(\tr(\alpha_s \alpha_s^\top ) -|\alpha_s^\top  e_s|^2)  \bigg)\d s \\ 
   &\hspace{1.72in}+ 
   \int_0^t I_{\{|Z_s|^{2}\geq \e\}} \sc(U_s)\langle e_s,\Delta_s \d B _s^1 \rangle + \int_0^t I_{\{|Z_s|^{2}\geq \e\}} \rc(U_s) \langle e_s, \alpha_s\d B _s^2\rangle  \\
    & \to \bigg[\!\int_0^t  I_{\{|Z_s| \neq 0\}} \!\bigg(\langle e_s,\beta_s \rangle + \frac{\sc^2(U_s)}{2r_s}(\tr(\Delta_s \Delta_s^\top ) -|\Delta_s^\top  e_s|^2)  + \frac{\rc^2(U_s)}{2r_s}(\tr(\alpha_s \alpha_s^\top ) -|\alpha_s^\top  e_s|^2)  \bigg)\d s \\ 
   &\hspace{1.72in}+
   \int_0^t I_{\{|Z_t| \neq 0\}}  \sc(U_s)\langle e_s,\Delta_s \d B _s^1 \rangle + \int_0^t  I_{\{|Z_t| \neq 0\}} \rc(U_s) \langle e_s, \alpha_s\d B _s^2 \rangle
\end{align*} a.s. as $\e_{k} \to 0$. This, together with  \eqref{e5lem3.1proof} and  \eqref{e6lem3.1proof},
gives \eqref{radial process}.
\end{proof}

\begin{proof}[Proof of Lemma \ref{A bound}] For simplicity of notation, we  denote $\sigma = \sigma(y)$,  $u = u(x,y), \Delta = \Delta(x,y)$, and $\alpha = \alpha(x,y)$ in this proof.
    Noting $\alpha=\Delta+2\sigma uu^\top $,
    we have
    \begin{align*}
        \tr(\alpha \alpha^\top )-|\alpha^\top e|^2=&\tr(\Delta \Delta^\top )+4\tr(\Delta(\sigma uu^\top )^\top )+\frac{4}{|\sigma^{-1}z|^2}\tr(zz^\top )-|\Delta^\top  e+2uu^\top \sigma^\top e|^2\\
        =& \tr(\Delta \Delta^\top )+\frac{4}{|\sigma^{-1}z|^2}\tr(\Delta(zz^\top (\sigma^{-1})^\top )^\top )+\frac{4}{|\sigma^{-1}z|^2}|z|^2\\
        &-|\Delta^\top e|^2-4\tr(\Delta e(uu^\top \sigma^\top e)^\top )-4|uu^\top \sigma^\top e|^2\\
        =& \tr(\Delta \Delta^\top )+\frac{4}{|\sigma^{-1}z|^2}\tr(\Delta^\top (zz^\top (\sigma^{-1})^\top ))+\frac{4}{|\sigma^{-1}z|^2}|z|^2\\&-|\Delta^\top e|^2-4\tr(\Delta^\top  zz^\top (\sigma^{-1})^\top )-\frac{4}{|\sigma^{-1}z|^2}4|z|^2\\
        =&\tr(\Delta \Delta^\top )-|\Delta^\top  e|^2,
    \end{align*}
    where we used the definitions $u=\sigma^{-1}z/|\sigma^{-1}z|$, and $e=z/|z|$.

    Assumption \ref{sigma lipschitz} gives us \begin{align*}
  \big| \tr(\Delta \Delta^\top )-|\Delta^{\top } e|^2| \le |\Delta|^{2} + |\Delta^{\top } e\big|^2 \le 2 |\Delta|^{2} \le 2 d\|\Delta\|^{2} \le 2 d L_{2} |x-y|^{2}.
\end{align*} The proof is complete.
    \end{proof}

\begin{proof}[Proof of Lemma \ref{nonexplosive condition}]
First note that Assumption \ref{b growth} gives us $|b(0,\delta_0)| \le L_{3}<\infty$, where $\delta_0$ is the Dirac measure at $0$.
Denote $X_t=X_t(0)$ throughout the proof.  
Then   It\^o's formula  gives
\begin{align*}
\d(|X_t|^2)=\left(2 \langle X_t, b(X_t,\mu_t)\rangle + \tr (\sigma(X_t)\sigma(X_t)^\top )\right)\d t + 2\langle X_t, \sigma(X_t)\d W_t \rangle.
\end{align*}
Consider  the function $g\in C^2([0,\infty))$
\begin{align*}
    g(s) : = \begin{cases}
  \sqrt s,    & \text{ if  } s \ge 1, \\
  -\frac{1}{8}s^{2} + \frac{3}{4} s + \frac{3}{8},   & \text{ if } 0\le s < 1,
\end{cases}
\end{align*}
which has derivatives
\begin{align*}
    g'(s) = \begin{cases}
    \frac{1}{2}s^{-1/2}, & \text{ if  } s \geq 1, \\
    -\frac{1}{4}s+\frac{3}{4}, & \text{ if } 0\leq s <1
    \end{cases}
\quad \text{ and }\quad
    g''(s)= \begin{cases}
    -\frac{1}{4}s^{-3/2}, & \text{ if  } s \geq 1, \\
    -\frac{1}{4}, & \text{ if  } 0\leq s < 1.
    \end{cases}
\end{align*}
Put   $\Tilde{e}_t=\frac{X_t}{|X_t|}$ for $X_{t} \neq 0$.  Then we have by It\^o's formula that
\begin{align*}
    \d&(g(|X_t|^2)) \\ &= I_{\{|X_t|<1\}}\left(\left(-\frac{|X_t|^2}{4}+\frac{3}{4}\right)\left(2\langle X_t,b(X_t,\mu_t)\rangle +\tr (\sigma(X_t)\right)\sigma(X_t)^\top )-\frac{1}{2}|\sigma(X_t)^{\top }X_t|^2\right)\d t \\
    &\ \ +I_{\{|X_t|\geq1\}}\left(\langle \Tilde{e}_t, b(X_t,\mu_t)\rangle +\frac{1}{2|X_t|}\tr (\sigma(X_t)\sigma(X_t)^\top )-\frac{1}{2}|\sigma(X_t)^\top  \Tilde{e}_t|^2\right)\d t
    +\d M_t,
\end{align*}
where 
 $M_t$ is a martingale.  Then, it follows that 
\begin{align}\label{e:E[g]differential}
\nonumber     \lefteqn{   \frac{\d}{\d t}\E[g(|X_t|^2)]} \\ \nonumber &= \E\left[ I_{\{|X_t|<1\}}\left(\left(-\frac{|X_t|^2}{4}+\frac{3}{4}\right)\left(2\langle X_t,b(X_t,\mu_t)\rangle +\tr (\sigma(X_t)\right)\sigma(X_t)^\top )-\frac{1}{2}|\sigma(X_t)X_t|^2\right)\right] \\
\nonumber    &\ \ +\E\left[I_{\{|X_t|\geq1\}}\left(\langle \Tilde{e}_t, b(X_t,\mu_t)\rangle +\frac{1}{2|X_t|}(\tr (\sigma(X_t)\sigma(X_t)^\top )-\frac{1}{2}|\sigma(X_t)^\top  \Tilde{e}_t|^2)\right)\right]\\
   &   \leq\E\left[I_{\{|X_t|\geq 1\}} \langle \Tilde{e}_t, b(X_t,\mu_t)\rangle + I_{\{|X_t|<1\}}\frac{3-|X_t|^2}{2}\langle X_t,b(X_t,\mu_t) \rangle +K_0\right],
\end{align}
where $K_0=\sup\{\tr(\sigma(x)\sigma(x)^\top )\}$,  
  which  is finite by Assumption \ref{bound 1}.  Now by Assumption \ref{b lipschitz}, we have
\begin{align*}
     \langle X_t, b(X_t,\mu_t)\rangle  
     &=\langle X_t,b(0,\delta_0)\rangle + \langle X_t, b(X_t,\mu_t)-b(0,\delta_0)\rangle\\
     &\leq \kappa(|X_t|)|X_t|^2+L_1 \W_1(\mu_t,\delta_0)|X_t|+\langle X_t,b(0,\delta_0)\rangle\\
     &\leq \kappa(|X_t|)|X_t|^2+L_1\E[|X_t|]|X_t|+|X_t||b(0,\delta_0)|.
\end{align*}
Thus
\begin{equation*}
    I_{\{|X_t|\geq 1\}} \langle \Tilde{e}_t, b(X_t,\mu_t)\rangle \leq I_{\{|X_t|\geq 1}\kappa(|X_t|))|X_t|+L_1\E[|X_t|]+|b(0,\delta_0)|,
\end{equation*}
and
\begin{align*}
    I_{\{|X_t|<1\}}\frac{3-|X_t|^2}{2}\langle X_t,b(X_t,\mu_t) \rangle& \leq   I_{\{|X_t|<1\}}\frac{3|X_t|-|X_t|^3}{2} (\kappa(|X_t|)|X_t|+L_1\E[|X_t|]+|b(0,\delta_0)|)\\ &\le I_{\{|X_t|<1\}}(\kappa_1+|b(0,\delta_0)|+L_1\E[|X_t|]),
\end{align*}
where  $\kappa_1=\sup_{x\in [0,1]}|\kappa(x)|$; note also that we used the fact that $3x-x^3\leq 2$ 
for $x\in[0,1]$  in the last inequality.  Now plugging these observations into \eqref{e:E[g]differential} leads to
\begin{align*}
    \frac{\d}{\d t}\E[g(|X_t|^2)] & \leq  \E\left[I_{\{|X_t|\geq 1\}}\left(\kappa(|X_t|)|X_t|+L_1\E[|X_t|]+|b(0,\delta_0)|\right)\right] \\
    & \quad + \E\left[I_{\{|X_t|\leq1\}} \left(\kappa_1+|b(0,\delta_0)|+L_1\E[|X_t|]\right)\right]  +K_0 \\
 &   \leq  \E\left[I_{\{|X_t|\geq 1\}}\kappa(|X_t|)|X_t| +L_1\E[|X_t|]+C_1\right],
\end{align*}
where $C_1=K_0+\kappa_1+|b(0,\delta_0)|$.  Now using $R_0$ from \eqref{nonexplosive condition} we can define    $\kappa_{R_0}=\sup_{x\in[0,R_0]}|\kappa(x)|$ and $C_2=C_1+\kappa_{R_0}R_0$.  Then
\begin{align*}
    \frac{\d}{\d t}\E[g(|X_t|^2)]&\leq \E\left[I_{\{|X_t|>R_{0}\}}(-L_1-K)|X_t|+\kappa_{R_{0}}R_{0}+L_1\E[|X_t|]+C_1\right]\\ & \leq\E\left[I_{\{|X_t|>R_{0}\}}(-L_1-K)|X_t|+L_1\E[|X_t|]+C_2\right]  \\
    &= \E\left[(-L_1-K)|X_t|+I_{\{|X_t|\leq R_{0}\}}(L_1+K)|X_t|+L_1\E[|X_t|]+C_2\right]  \\
    &\leq \E\left[(-L_1-K)|X_t|+I_{\{|X_t|\leq R_{0}\}}(L_1+K)R_{0}+L_1\E[|X_t|]+C_2\right] \\
    & \le -K\E[|X_t|]+C_3,
\end{align*}
where  $C_3=C_2+(L_1+K_1)R_{0}$.
Furthermore, using the elementary inequality $x \ge -\frac18 x^{4} +\frac34 x^{2} $ for $x\in [0,1]$, we have
\begin{align*} |X_t| &=I_{\{|X_t|<1\}}|X_t|+I_{\{|X_t|\geq 1\}}|X_t| \\ &\geq I_{\{|X_t|<1}\left(-\frac{1}{8}|X_t|^4+\frac{3}{4}|X_t|^2\right)+I_{\{|X_t|\geq1\}}\sqrt{|X_t|^2}\\ & \geq g(|X_t|^2)-\frac{3}{8}. \end{align*}
Consequently, it follows that
\[\frac{\d}{\d t}\E[g(|X_t|^2)]\leq -K\E[g(|X_t|^2)+C_4\]
where $C_4=C_3+\frac{3}{8}K$.  By Gronwall's inequality \cite[lemma 2.2]{ChenLi-89} we have
\[\E[g(|X_t|^2)]\leq \frac{C_4}{K}(1-e^{-Kt})<\infty,\ \ \text{ for }t\ge 0.\]
Last, we note that $x \le g(x^{2}) +1$ for all $x\in \R$ and hence
\[\sup_{t\ge 0}\E[|X_t|]\leq 1+\sup_{t\ge 0} \E[g(|X_t|^2)]<\infty.\]
    The proof is complete.
\end{proof}

\def\cprime{$'$} \def\polhk#1{\setbox0=\hbox{#1}{\ooalign{\hidewidth
  \lower1.5ex\hbox{`}\hidewidth\crcr\unhbox0}}}

\end{document}